\documentclass[12pt,reqno]{amsart}
\usepackage{amsmath,amssymb,mathrsfs}
\usepackage{color}
\usepackage[margin=1in]{geometry}
\usepackage{times}
\usepackage{enumitem} 
\usepackage[colorlinks=true,citecolor=magenta,linkcolor=blue]{hyperref}
\allowdisplaybreaks

\newcommand{\vv}{\mathbf v}
\newcommand{\calA}{\mathcal{A}}
\renewcommand{\O}{\Omega}
\newcommand{\R}{\mathbb R}
\newcommand{\pa}{\partial}
\newcommand{\di}{\mathrm{div}}

\newcommand{\na}{\nabla}
\newcommand{\dx}{dxdt}

\renewcommand{\div}{\text{div}}
\newcommand{\eps}{\varepsilon}
\newcommand{\ue}{u_\eps}
\newcommand{\uee}{u_{\eps'}}
\newcommand{\Ge}{g_\eps}
\newcommand{\fe}{f_\eps}

\newcommand{\LQ}[2]{L^{#1}(0,T;L^{#2}(\Omega_t))}
\newcommand{\LW}[3]{L^{#1}(0,T;W^{#2,#3}_0(\Omega_t))}
\newcommand{\intO}{\int_{\Omega_t}}

\renewcommand{\varrho}{\vartheta}

\newtheorem{theorem}{Theorem}[section]
\newtheorem{definition}{Definition}[section]
\newtheorem{lemma}{Lemma}[section]

\newtheorem{remark}{Remark}[section]

\title[Quasilinear equations in moving domains and $L^1$ data]{Quasilinear parabolic equations with first order terms and $L^1$-data in moving domains}
\author[D. Lan]{Do Lan}
\address{Do Lan\hfill\break
Department of Mathematics, Thuyloi University, 175 Tay Son Street, Hanoi, Vietnam}
\email{dolan@tlu.edu.vn}
\author[D.T. Son]{Dang Thanh Son}
\address{Dang Thanh Son\hfill \break
Foundation Sciences Faculty, Telecommunications University, 101 Mai Xuan Thuong, Nha Trang, Vietnam}
\email{dangthanhson@tcu.edu.vn}
\author[B.Q. Tang]{Bao Quoc Tang}
\address{Bao Quoc Tang \hfill\break
	Institute of Mathematics and Scientific Computing, University of Graz, 
	Heinrichstrasse 36, 8010 Graz, Austria}
\email{quoc.tang@uni-graz.at, baotangquoc@gmail.com}
\author[L.T. Thuy]{Le Thi Thuy}
\address{Le Thi Thuy \hfill\break
Department of Mathematics, Electric Power University, 235 Hoang Quoc Viet, Hanoi, Vietnam}
\email{thuylt@epu.edu.vn, thuylephuong@gmail.com}
\begin{document}
	\begin{abstract}
		The global existence of weak solutions to a class of quasilinear parabolic equations with nonlinearities depending on first order terms and integrable data in a moving domain is investigated. The class includes the $p$-Laplace equation as a special case. Weak solutions are shown to be global by obtaining appropriate estimates on the gradient as well as a suitable version of Aubin-Lions lemma in moving domains.
	\end{abstract}
	\maketitle
	
\section{Introduction}
Problems defined on a domain which changes its shape in time have recently attracted a lot of attention from mathematical community since not only they lead to interesting mathematical questions, but also
they arise naturally in physics, biology, chemistry and many other fields. Examples include the studies of pattern formation on evolving surfaces \cite{BEM11,GMMSV11}, of surfactants in
two-phase flows \cite{GLS14}, of dealloying by surface dissolution of a binary alloy \cite{EE08}, of a diffusion interface model for linear surface partial differential equations \cite{ES09}, or of modeling and simulation of cell mobility \cite{ESV12}. We refer the interested reader to the extensive review paper \cite{KK15}. In this paper, we study the global existence of a quasilinear parabolic equation in moving domains, i.e. domains with shapes evolving in time. The equation contains a quasilinear diffusion operator, which includes the $p$-Laplacian as a special case, has a nonlinearity depending on the zero and first order terms, and has external force and initial data which are only integrable.


\medskip
To precisely state the problem under consideration, we consider a bounded domain $\O_0\subset \R^d$, $d\geq 1$, with smooth boundary $\pa\O_0$. Let $\vv: \R^d\times \R \to \R^d$ be a smooth and compactly supported vector field and $\zeta: \R^d \times \R \to \R^d$ be its corresponding flow, i.e. $\zeta$ solves
\begin{equation*}
	\pa_t\zeta(x,t) = \vv(\zeta(x,t),t), \qquad \zeta(x_0,0) = x_0
\end{equation*}
for any $x_0 \in \R^d$. Note that for each fixed $x$, the mapping $t\mapsto \zeta_t(x)$ is an integral curve of $\vv$ and for fixed $t$, the mapping $x_0 \mapsto \zeta_t(x_0)$ is a diffeomorphism. Assuming that $\O_0 \subset \mathrm{supp}(\vv)$, we define $\Omega_t = \zeta_t(\O_0)$ and the non-cylindrical domain as
\begin{equation*}
	Q_T:= \bigcup_{t\in (0,T)}\O_t\times \{t\} = \bigcup_{t\in (0,T)}\zeta_t(\O_0)\times \{t\},
\end{equation*}
\begin{equation*}
	\Sigma_T:= \bigcup_{t\in (0,T)}\pa\O_t\times \{t\} = \bigcup_{t\in (0,T)}\zeta_t(\pa\O_0)\times\{t\}.
\end{equation*}
We choose an open and bounded subset $\widehat{\O}$ of $\R^d$ such that $\cup_{t\in[0,T]}\Omega_t \subset \widehat{\O}$ and let $\widehat{Q}:=\widehat{\O}\times (0,T)$.

We also need to define time-space spaces in moving domains. Let $\{X(t)\}_{t\in [0,T]}$ be a family of Banach spaces, then we define 
\begin{equation*}	
	L^p(0,T;X(t)) = \{f:Q_T\to \mathbb R:\; f(t)\in X(t) \text{ for a.e. } t\in (0,T)\}
\end{equation*}
with the norm
\begin{equation*}
	\|f\|_{L^p(0,T;X(t))} = \left(\int_0^T\|f(t)\|_{X(t)}^pdt\right)^{\frac 1p} <+\infty.
\end{equation*}
Very common in this paper we use $X(t) = L^q(\Omega_t)$ or $X(t) = W_0^{1,q}(\Omega_t)$. In particular, when $p=q$, then we write simply $L^p(Q_T)$ instead of $L^p(0,T;L^p(\Omega_t))$.

\medskip
	The main goal of the present paper is to study the global existence of the following quasilinear problem
	\begin{equation}\label{e1}
		\begin{cases}
			\partial_t u - \div (a(x,t,\nabla u)) + \div (u\vv) + g(x,t,u, \nabla u) = f, &(x,t)\in Q_T,\\
			u(x,t) = 0, &(x,t)\in \Sigma_T,\\
			u(x,0) = u_0(x), &x\in \O_0,
		\end{cases}
	\end{equation}
	with the external force $f\in L^1(Q_T)$ and initial data $u_0 \in L^1(\O_0)$. The nonlinear diffusion $a$ is assumed to satisfy
	\begin{enumerate}[label=(A\theenumi),ref=A\theenumi] 
	\item\label{A0} $a: \widehat{Q}\times \R^d \to \R^d$ is a Carath\'eodory function; 
	
	\item\label{A2} there exists $p>\frac{2d+1}{d+1}$ such that, for $(x,t)\in \widehat{Q}$ and $\xi \in \R^d$,
	\[	
		|a(x,t,\xi)| \leq \varphi(x,t) + K|\xi|^{p-1}
	\]
	where $\varphi \in L^{p'}(\widehat{Q})$, $1/p+1/p' = 1$ and $K\geq 0$;
	
	\item\label{A3} there exists $\alpha>0$ such that
	\[
		a(x,t,\xi)\xi \geq \alpha|\xi|^p,
	\]
	where $(x,t)\in \widehat{Q}$ and $\xi\in\R^d$;

	\item\label{A1} for almost $(x,t)\in \widehat{Q}$ and all $\xi, \xi' \in \R^d$,
	\[
		(a(x,t,\xi) - a(x,t,\xi'))(\xi - \xi') \geq \frac{1}{\Theta(x,t,\xi,\xi')}|\xi-\xi'|^\theta \quad \text{ and } \quad a(x,t,0) = 0;
	\]			
	for $\theta>1$, and $\Theta$ is a nonnegative function which satisfies
	\begin{equation}\label{beta}
		|\Theta(x,t,\xi,\xi')| \leq C(1+ |\xi| + |\xi'|)^{\varrho}
	\end{equation}
	where 
	\begin{equation}\label{varrho}
		\varrho < (\theta-1)\left(p-\frac{d}{d+1}\right).
	\end{equation}
	\end{enumerate}		
	The nonlinearity $g: \widehat{Q}\times \mathbb R\times \mathbb R^d \to \mathbb R$ satisfies: $g$ is continuous with respect to the third and fourth variables, and
		\begin{enumerate}[label=(G\theenumi),ref=G\theenumi] 
		\item\label{G1} it holds
		\begin{equation*}
			\lambda g(x,t,\lambda, \xi) \geq 0 \quad \text{ for all } \quad \lambda\in \R, \xi \in \R^d;
		\end{equation*}
		
		\item\label{G2}
		$g$ has a {\it subcritical} growth on the gradient, i.e. there exists $0 \leq \sigma < p$ such that
		\begin{equation*}
			|g(x,t,\lambda,\xi)| \leq h(|\lambda|)(\gamma(x,t) + |\xi|^\sigma)
		\end{equation*}
		where $\gamma \in L^1(\widehat{Q})$, and $h$ is an increasing function in $\mathbb R_+$.
		\end{enumerate}	
	
	Let us briefly discuss about the above conditions. The conditions \eqref{A0}--\eqref{A1} of $a$ assure that it contains the important case of $p$-Laplacian, i.e. $$a(x,t,\xi)\equiv a_p(\xi) = |\nabla \xi|^{p-2}\nabla \xi \quad \text{ for } \quad p>\frac{2d+1}{d+1}.$$ Moreover, the technical condition \eqref{A1} is weaker than the usual {\it strong monotonicity} condition
	\begin{equation*}\label{strong_monotone}
		(a(x,t,\xi) - a(x,t,\xi'))(\xi - \xi') \geq C|\xi - \xi'|^p,
	\end{equation*} 
	for some $C>0$, but still stronger than the mere monotone property condition, i.e.
	\begin{equation}\label{mon}
	(a(x,t,\xi) - a(x,t,\xi'))(\xi - \xi') \geq 0.
	\end{equation} 	 
	We also remark that the condition \eqref{G2} allows $g$ to have arbitrary growth in the zero order term, as long as it has the suitable sign stated in \eqref{G1}. A typical example of $g$ is
	\begin{equation*}
		g(x,t,u,\na u) = Cu^{2k+1}(\gamma(x,t) + |\na u|^{\sigma})
	\end{equation*}
	where $k\in \mathbb N$ is arbitrary and $0\leq \sigma < p$. 
	
	\medskip
	Elliptic or parabolic equations with irregular data such as $L^1$ or bounded measure appear frequently in applications and therefore are of interest and importance. Concrete examples include elliptic systems modeling electronical devices \cite{GH94}, the Fokker-Planck equation arising from populations dynamics \cite{GS98}, models of turbulent flows in oceanography and climatology \cite{Lew97}, incompressible flows with small Reynolds number \cite{Lio96}, or Keller-Segel or Shigesada-Kawasaki-Teramoto type systems \cite{winkler2019role}. Global existence of weak or renormalized solutions to special cases of \eqref{e1} in fixed domains has been studied extensively in the literature. Let us mention several related works: in \cite{BG89,boccardo1997nonlinear}, the authors considered \eqref{e1} where conditions \eqref{A0}--\eqref{A1} are imposed, but the function $g$ is either zero or does not have the first order term; similar results were shown in \cite{Bla93} assuming \eqref{mon} instead of \eqref{A1}; the case when $g$ depends on the first order term was considered in e.g. \cite{GS01, andreu2002existence}, but the second order term therein is a linear elliptic operator, for instance $\mathrm{div}(a(x,t,\na u)) = \Delta u$; when $p>1$ arbitrary, one can show global renormalized solutions \cite{BM97} (see also Remark \ref{remark1}). Related results are also obtained for systems without the first order terms \cite{BS05}.
	
	\medskip
	The global existence of solutions to \eqref{e1} in moving domains with $L^1$-data, up to our knowledge, is not studied therefore is the main motivation of our paper. We would like also to emphasize that, even in the case of a fixed domain, our results extend that of \cite{BG89} and \cite{GS01}. Related results for quasilinear parabolic problems in time-dependent domains can also be found in e.g. \cite{CNO17,bogelein2018existence}.
		
	\medskip
	The main goal of this paper is to prove the global existence of a weak solution to \eqref{e1} under the conditions \eqref{A0}--\eqref{A1}, \eqref{G1}--\eqref{G2} and data $f\in L^1(Q_T)$ and $u_0\in L^1(\Omega_0)$. To state the main result, we first give the precise definition of a weak solution.
	\begin{definition}[Weak solutions]\label{def:sol}
		Let $T>0$ be arbitrary. 
		A function
		\begin{equation*}
			u\in C([0,T];L^1(\O_t))\cap \left(\underset{1\leq q < p - \frac{d}{d+1}}{\bigcap}L^q(0,T;W_0^{1,q}(\O_t))\right)
		\end{equation*}
		is called a weak solution to \eqref{e1} on $(0,T)$ if $g(x,t,u,\nabla u)\in L^1(Q_T)$ and for all test functions $\psi\in C([0,T];W^{1,\infty}_0(\Omega_t))\cap C^1((0,T);L^{\infty}(\Omega_t))$, the following weak formulation holds
		\begin{equation*}
			\begin{aligned}
			&\int_{\Omega_T}u(T)\psi(T)dx  - \int_0^T\int_{\O_t}u\psi_t \dx\\
			&+ \int_0^T\int_{\O_t}[a(x,t,\na u)\cdot \na \psi - u\vv\cdot \na \psi + g(x,t,u,\nabla u)\psi]\dx\\
			&= \int_{\O_0}u_0\psi(0)dx + \int_0^T\int_{\O_t}f\psi \dx.
			\end{aligned}
		\end{equation*}
	\end{definition}
	All the terms above are obviously well-defined except for the term containing $a(x,t,\na u)\cdot \na \psi$. From the growth assumption \eqref{A2} of $a$, and the fact that $u\in L^q(0,T;W_0^{1,q}(\Omega_t))$ for all $1\leq q < p - d/(d+1)$, it follows that $a(\cdot,\cdot,\na u)\in L^s(Q_T)$ for all $1\leq s < 1 + \frac{1}{(p-1)(d+1)}$, and therefore, the integration $\int_{Q_T}a(x,t,\na u)\cdot \na \psi dxdt$ makes sense since $\na \psi\in L^\infty(Q_T)$.
	\begin{remark}\label{remark1}
		The condition $p > (2d+1)/(d+1)$ is needed to define the weak solution. When $p \leq (2d+1)/(d+1)$, we can only obtain $\nabla u \in L^q(Q_T)^d$ for $q\in (0,1)$. In this case, one can either show the existence of renormalized solutions, see e.g. \cite{BM97}, or weak solutions belonging to $L^r(0,T;W^{1,q}_0(\Omega_t))$ for $r,q$ are different, see e.g. \cite{boccardo1997nonlinear} where \eqref{e1} was studied in a cylindrical domain for all $p>1$ but without the nonlinearity $g$. These two directions go beyond the scope of this paper and therefore are left for upcoming research.
	\end{remark}
	The main result of this paper is the following theorem.
	\begin{theorem}[Global existence of weak solutions]\label{thm:main}
	    Assume that the vector field $\vv\in C^1(\R^d,C^0(\R))$. 
		Assume the conditions \eqref{A0}--\eqref{A1} and \eqref{G1}--\eqref{G2}. Then for any $u_0 \in L^1(\O_0)$ and any $f\in L^1(Q_T)$, there exists a global weak solution $u$ to \eqref{e1} on $(0,T)$ as in Definition \ref{def:sol}.
	\end{theorem}

	Let us describe the main ideas in proving Theorem \ref{thm:main}. To treat moving domains, one can transform the problem into the case of fixed domains and then study the new equation, with the cost of some additional terms. Usually these additional terms depend significantly on the problem itself, and therefore each problem needs to be treated separately. As an attempt to have a more unified mechanism, a different approach is to derive a mechanism to work on the moving domains directly, that is to establish parallel tools for moving domains corresponding to that of fixed domains. This research direction has been investigated by many authors (see e.g. \cite{AES15,AET18,MB08,Vie14}).
	
	In this paper, we adapt the second approach to prove Theorem \ref{thm:main}, meaning that we treat \eqref{e1} directly on the non-cylindrical domain $Q_T$. More precisely, first, we consider an approximation of \eqref{e1} in which the data is approximated by $f_\eps \in L^\infty(Q_T)$ and by $u_0\in L^\infty(\Omega_0)$. Moreover, we also regularize the nonlinearity $g_\eps = g(1+\eps|g|)^{-1}$ which is bounded for any fixed $\eps>0$. Thanks to this regularization, we can use the method from \cite{CNO17} to obtain the global existence of an approximate solution $\ue$. The next goal is to derive estimates of this approximate solution uniformly in $\eps$. In order to do that, due to the low regularity of the data, we refine the analysis in \cite{GS01} to adapt to the case of quasilinear problem \eqref{e1}. Once the uniform estimates for $\ue$ are obtained, we would like to pass to the limit as $\eps\to 0$, which consequently requires an Aubin-Lions lemma in the case of moving domains. A similar lemma has been shown in different works (see e.g. \cite{Mou16} or \cite{Fuj70}), but they are not applicable to our situation. Therefore, we prove a new Aubin-Lions lemma in moving domains, which allows us to first obtain the almost everywhere convergence $\ue \to u$ and then consequently $\|\ue - u\|_{L^1(Q_T)} \to 0$. Due to the dependence of the nonlinearity on $\na u$, this convergence is not yet enough. By using the ideas from \cite{GS01}, we utilize the assumptions \eqref{G1} and \eqref{G2} to show that the convergence $\na \ue \to \na u$ holds almost everywhere. This in turn helps to get $g_\eps(x,t,\ue,\na \ue) \to g(x,t,u,\na u)$ and $a(x,t,\na\ue) \to a(x,t,\na u)$ in appropriate spaces, and eventually to obtain $u$ to be a weak solution to \eqref{e1}.
	
	\medskip
	{\bf The rest of this paper is organized as follows:} In the next Section, we derive uniform a-priori estimates for approximate solutions, which are needed to pass to the limit in Section \ref{proof} to obtain the weak solution of \eqref{e1}. The Appendix \ref{appendix1} and \ref{appendix2} provide the existence of an approximate solution and a proof of the Aubin-Lions lemma in moving domains respectively.

	\medskip
	{\bf Notation.} We will use in this paper the following set of notations.
	\begin{itemize}
		\item Recall that we simply write $L^p(Q_T)$ instead of $\LQ{p}{p}$.
		\item The double integration $\int_0^T\int_{\Omega_t}dxdt$ is written using the shorthand notation $\int_{Q_T}dxdt$.
		\item We usually write $C = C(\alpha, \beta,\gamma,\ldots)$ to indicate that the constant $C$ depends on the arguments $\alpha, \beta, \gamma$, etc.
		\item As we will use it frequently in the paper, for fixed $T>0$ we write
		\begin{equation*}
		    \|\vv\|_\infty:= \|\vv\|_{L^\infty(Q_T)} \quad \text{ and } \quad \|\div \vv\|_{\infty}:= \|\div \vv\|_{L^\infty(Q_T)},
		\end{equation*}
		which are well-defined when $\vv\in C^1(\R^d,C^0(\R))$.
	\end{itemize}
	
\section{Uniform estimates}\label{approximate}
In this section, we consider an approximate problem to \eqref{e1} and derive uniform {\it a priori} estimates for the approximate solution. These estimates play a crucial role in passing to the limit to obtain a weak solution to \eqref{e1}. For simplicity we write $g(u,\na u)$ instead of $g(x,t,u,\na u)$.

\medskip
Fix an arbitrary time horizon $T>0$. As usual we regularize the initial data $u_0$ and the external term $f$ by more regular data $u_{0,\varepsilon}\in L^\infty(\Omega_0)$ and $f_\varepsilon \in L^{\infty}(Q_T)$ for $\varepsilon>0$, such that
\begin{equation}\label{u0f}
	\lim_{\eps \to 0}\|u_{0,\eps} - u_0\|_{L^1(\Omega_0)} = 0 \quad \text{ and } \quad \lim_{\eps \to 0}\|f_{\eps} - f\|_{L^1(Q_T)} = 0,
\end{equation}
and
\begin{equation}\label{increasing}
	\|u_{0,\eps}\|_{L^1(\Omega_0)} \leq \|u_0\|_{L^1(\Omega_0)} \quad \text{ and } \|\fe\|_{L^1(Q_T)} \leq \|f\|_{L^1(Q_T)}.
\end{equation}
Moreover, we also regularize the nonlinear first order term by a bounded nonlinearity, namely, for $\eps > 0$,
\begin{equation*}
\Ge(w,\na w):= \frac{g(w,\na w)}{1+\varepsilon|g(w,\na w)|}.
\end{equation*}
Note that for any fixed $\varepsilon>0$, we have
\begin{equation*}
	|\Ge(w,\na w)| \leq \frac 1\eps \quad \text{ for all } \quad (x,t)\in Q_T\quad \text{ and all } \quad w.
\end{equation*}
The approximate problem reads as, 
\begin{equation}\label{approx}
	\begin{cases}
		\pa_t\ue - \div (a(x,t,\na \ue)) + \div(\ue\vv) + g_\eps(\ue,\na \ue) = f_\eps, &(x,t)\in Q_T,\\
		\ue(x,t) = 0, &(x,t)\in \Sigma_T,\\
		\ue(x,0) = u_{0,\varepsilon}(x), &x\in \Omega_0.
	\end{cases}
\end{equation}

\begin{definition}[Weak solutions to \eqref{approx}]
	A weak solution to \eqref{approx} on $(0,T)$ is a function $\ue\in C([0,T];L^p(\Omega_t))\cap L^p(0,T;W_0^{1,p}(\Omega_t))$ with $\partial_t\ue \in L^{p'}(0,T;W^{-1,p'}(\Omega_t))$, where $W^{-1,p'}(\Omega_t) = (W_0^{1,p}(\Omega_t))^*$, such that
	\begin{align*}
		\int_0^T\langle \partial_t \ue, \phi \rangle_{W^{-1,p'}, W_0^{1,p}}dt + \int_0^T\int_{\Omega_t}a(x,t,\na\ue)\cdot \na \phi dxdt\\
		-\int_0^T\int_{\O_t}\ue \vv \cdot\na\phi dxdt + \int_0^T\int_{\Omega_t}g_\eps(\ue,\na\ue)\phi dxdt = \int_0^T\int_{\Omega_t}\fe \phi dxdt
	\end{align*}
	for all test function $\phi \in \LW{p}{1}{p}$.
\end{definition}

The global existence of a weak solution to \eqref{approx} can be obtained by the slicing technique in e.g. \cite{CNO17} with suitable, slight modifications. For the sake of completeness, we sketch the main steps of the proof, and postpone it to the Appendix \ref{appendix1} in order to not interrupt the train of thought. 
\begin{theorem}[Existence of a global solution to the approximate problem]\label{thm:approximate}
Fix $T>0$. For any $u_{0,\eps}\in L^\infty(\Omega_0)$ and $\fe\in  L^\infty(Q_T)$, there exists a weak solution to \eqref{approx} on $(0,T)$.
\end{theorem}
The focus of this section is therefore to obtain a-priori estimates of solutions to \eqref{approx} which are {\it uniform in $\varepsilon$.} We divide the section further into two subsections, in which the first one shows uniform bounds of approximate solutions in Sobolev spaces, while the second provides uniform bounds of the nonlinearity $\Ge(\ue,\na\ue)$.

\subsection{Uniform bounds of approximate solutions}
The following lemma is the main result of this subsection.
\begin{lemma}\label{UniformBounds}
	There exists a constant $C(T)$ depending on $T, \vv$, $\|u_0\|_{L^1(\O_0)}$ and $\|f\|_{L^1(Q_T)}$ but {\normalfont independent of $\eps$} such that
			\begin{equation*}
				\|\ue \|_{\LW{q}{1}{q}} \leq C(T)
			\end{equation*}
			for all $1 \leq q < p - d/(d+1)$.
		\end{lemma}
The proof of this lemma is long and technical and is therefore divided into several steps. As a preparation, we need a lemma about Sobolev embeddings in moving domains.
\begin{lemma}[Sobolev embeddings]\label{embedding}
	Fix $T>0$. Then there exists a constant $C_{\Omega,T}$ depending on $T$ and $\vv$ such that
	\begin{equation}\label{desired}
		\|u\|_{L^{q^*}(\Omega_t)} \leq C_{\Omega,T}\|\nabla u\|_{L^q(\Omega_t)} \quad \text{ for all } t\in [0,T]\quad  \text{ and } \quad u\in W_0^{1,q}(\Omega_t)
	\end{equation}
	where $q<d$ and 
	\begin{equation*}
		q^* = \frac{dq}{d-q}.
	\end{equation*}
\end{lemma}
\begin{proof}
	The classical Sobolev embedding gives
	\begin{equation*}
		\|u\|_{L^{q^*}(\Omega_0)} \leq C(\Omega_0)\|\nabla u\|_{L^q(\Omega_0)}.
	\end{equation*}
	Since $\zeta \in C([0,T];C^1(\mathbb R^d))$, there exists $a(T), b(T)$ such that
	\begin{equation*}
		a(T)\leq |\det(D\zeta_t)(x)| \leq b(T) \quad \text{for all} \quad t\in [0,T].
	\end{equation*}
	Now, for $t\in [0,T]$, we know that $\Omega_t = \zeta_t(\Omega_0)$. Therefore,
	\begin{equation*}
	\begin{aligned}
		\left(\int_{\Omega_t}|u(x)|^{q^*}dx\right)^{\frac{1}{q^*}} &= \left(\int_{\Omega_0}|u(\zeta_t(y))|^{q^*}|\det(D\zeta_t)|dy\right)^{\frac{1}{q^*}}\\
		&\leq b(T)^{\frac{1}{q^*}}C(\Omega_0)\left(\int_{\Omega_0}|\na u(\zeta_t (y))|^{q}dy\right)^{\frac 1q}\\
		&\leq b(T)^{\frac{1}{q^*}}a(T)^{-\frac 1q}C(\Omega_0)\left(\int_{\Omega_t}|\na u(x)|^{q}dx\right)^{\frac 1q}
	\end{aligned}
	\end{equation*}
	which proves the desired estimate \eqref{desired}.
\end{proof}
	\begin{lemma}\label{lem:inter}
				Assume that $\ue \in \LW{p}{1}{p}$ satisfies
				\begin{equation}\label{cond1}
					\sup_{t\in (0,T)}\intO |\ue| dx \leq \beta,
				\end{equation}
				and for each $n\in \mathbb N$,
				\begin{equation}\label{cond2}
					\int_{B_n}|\na \ue|^pdxdt \leq C_0 + C_1\int_{E_n}|\na \ue|dxdt
				\end{equation}
				for some $\beta, C_0, C_1 >0$ {\normalfont independent of $\eps$} where
				\begin{equation}\label{def_BnEn}
					B_n = \{(x,t)\in Q_T: n\leq |\ue(x,t)| \leq n+1 \} \; \text{ and } \; E_n = \{ (x,t)\in Q_T: |\ue(x,t)| > n+1 \}.
				\end{equation}
				Then there exists $C(T,p,q,\beta,C_0,C_1)$ depending on $T,p,q,\beta, C_0$ and $C_1$, but {\normalfont independent of $\eps$}, such that
				\begin{equation}\label{e2_1}
					\|\ue\|_{\LW{q}{1}{q}} \leq C(T,p,q,\beta,C_0,C_1)
				\end{equation}
				for all $1\leq q < p - \frac{d}{d+1}$.
		\end{lemma}
\begin{remark}
	We remark that since $p>q$, obviously $\ue\in \LW{q}{1}{q}$ follows immediately from $\ue\in \LW{p}{1}{p}$ and $\|\ue\|_{\LW{q}{1}{q}} \leq C(T)\|\ue\|_{\LW{p}{1}{p}}$. However, the essential role of \eqref{e2_1} is that the constant $C(T)$ therein is independent of $\eps$, while the norm $ \|\ue\|_{\LW{p}{1}{p}}$ might blow up as $\eps\to 0$.
\end{remark}
\begin{proof}[Proof of Lemma \ref{lem:inter}]
	Let $1 \leq q < p$ be arbitrary. From \eqref{cond2}, by using H\"{o}lder's inequality, we have
	\begin{equation}\label{e3}
	\begin{aligned}
	\int_{B_n} |\nabla \ue |^pdxdt &\leq C_0+C_1\Big(\int_{E_n}|\nabla \ue|^qdxdt\Big)^{1/q}|E_n|^{(q-1)/q}\\
	&\leq C_0+C_1\|\nabla \ue\|_{L^q(Q_T)}|E_n|^{(q-1)/q}.
	\end{aligned}
	\end{equation}
	Since $q < p$ we can use H\"{o}lder's inequality, inequality \eqref{e3}, and the elementary inequality $(a+b)^{q/p} \leq a^{q/p} + b^{q/p}$ for $a,b\geq 0$, to obtain
	\begin{equation}\label{e5}
	\begin{aligned}
	\int_{B_n} |\nabla \ue|^q dxdt &\leq |B_n|^{(p-q)/p}\Big(\int_{B_n}|\nabla \ue|^pdxdt\Big)^{q/p}\\
	&\leq |B_n|^{(p-q)/p}\Big(C_0^{q/p}+C_1^{q/p}\|\nabla \ue\|^{q/p}_{L^q(Q_T)}|E_n|^{(q-1)/p}\Big).
	\end{aligned}
	\end{equation}
	Let $r \geq 0$ be chosen later. We have, by using the definitions of $B_n$ and $E_n$,
	\begin{equation}\label{e6}
	\begin{cases}
	&|B_n| \leq \frac{1}{n^r}\int_{B_n}|\ue|^r dxdt,\\
	&|E_n| \leq \frac{1}{n^r}\int_{E_n} |\ue|^r dxdt \leq \dfrac{1}{n^r}\|\ue\|^r_{L^r(Q_T)}.
	\end{cases}
	\end{equation}
	Inserting \eqref{e6} into \eqref{e5} yields
	\begin{equation}\label{e7}
	\begin{aligned}
	\int_{B_n} |\nabla \ue|^q dxdt &\leq C_0^{q/p}\Big(\dfrac{1}{n}\Big)^{r(p-q)/p}\Big(\int_{B_n}|\ue|^rdxdt\Big)^{(p-q)/p}\\
	&+C_1^{q/p}\|\nabla \ue\|^{q/p}_{L^q(Q_T)}\|\ue\|^{r(q-1)/p}_{L^r(Q_T)}\Big(\frac{1}{n}\Big)^{r(p-1)/p}\Big(\int_{B_n}|\ue|^rdxdt \Big)^{(p-q)/p}.
	\end{aligned}
	\end{equation}
	Let $K\in \mathbb N$ be chosen later. We split $\|\nabla \ue\|^q_{L^q(Q_T)}$ as follows
	\begin{equation}\label{e8}
	\|\na\ue\|_{L^q(Q_T)}^q = \int_{Q_T} |\nabla \ue|^q dxdt =\sum_{n = 0}^K \int_{B_n} |\nabla \ue|^q dxdt + \sum_{n = K+1}^\infty \int_{B_n} |\nabla \ue|^q dxdt.
	\end{equation}
	Since $|B_n| \leq |Q_T|$ and $|E_n| \leq |Q_T|$, we simply evaluate the first term in the right hand side of \eqref{e8} using \eqref{e5} as follows 
	\begin{equation}\label{e9}
	\sum_{n = 0}^K \int_{B_n}|\nabla \ue|^q dxdt \leq (K+1)C_2\Big(1+\|\nabla \ue\|^{q/p}_{L^q(Q_T)}\Big),
	\end{equation}
	where $C_2 = \max\{C_0^{q/p}|Q_T|^{(p-q)/p}, C_1^{q/p}|Q_T|^{(q-1)/p}\}$.
	Using Young's inequality in \eqref{e8}-\eqref{e9}, we get
	\begin{equation}\label{e10}
	\|\nabla \ue\|^q_{L^q(Q_T)}\leq C(K)+2\sum_{n= K+1}^\infty \int_{B_n}|\nabla \ue|^q dxdt,
	\end{equation}
	where
	\begin{equation*}
		C(K) = 2\frac{p-1}{p}\cdot((K+1)C_2)^{\frac{p}{p-1}}\left(\frac 2p\right)^{\frac{1}{p-1}} + 2(K+1)C_2.
	\end{equation*}
	Note that the constant $C(K)$ tends to infinity as $K\to\infty$. It remains to proceed to the study of the series which appears on the right hand side of \eqref{e10}. 
	Applying H\"{o}lder's inequality on the series with exponents $p/(p-q)$ and $p/q$ and using \eqref{e7}, we have
	\begin{equation}\label{e11}
	\begin{aligned}
	&\sum_{n = K+1}^\infty \int_{B_n} |\nabla \ue|^q dxdt\\
	&\leq C_0^{q/p}\Big(\sum_{n=K+1}^\infty \dfrac{1}{n^{r(p-q)/q}}\Big)^{q/p}\Big(\sum_{n=K+1}^\infty \int_{B_n} |\ue|^r dxdt\Big)^{(p-q)/p}\\
	&\quad +C_1^{q/p}\|\nabla \ue\|_{L^q(Q_T)}^{q/p}\|\ue\|_{L^r(Q_T)}^{r(q-1)/p}\Big(\sum_{n=K+1}^\infty\dfrac{1}{n^{r(p-1)/q}}\Big)^{q/p}\Big(\sum_{K+1}^\infty \int_{B_n}|\ue|^r dxdt \Big)^{(p-q)/p}\\
	&\leq C_0^{q/p}\Big(\sum_{n=K+1}^\infty \dfrac{1}{n^{r(p-q)/q}}\Big)^{q/p}\|\ue\|_{L^r(Q_T)}^{r(p-q)/p}\\
	&\quad +C_1^{q/p}\|\nabla \ue\|^{q/p}_{L^q(Q_T)}\|\ue\|^{r(p-1)/p}_{L^r(Q_T)}\Big(\sum_{n=K+1}^\infty \dfrac{1}{n^{r(p-1)/q}}\Big)^{q/p}.
	\end{aligned}
	\end{equation}
	We choose $r$ so that the remainder of the series above converges to zero as $K \to \infty$, i.e.
	\begin{equation}\label{cond_r}
		\frac{r(p-q)}{q} > 1.
	\end{equation}
	Note that due to $q \geq 1$, this already implies $r(p-1)/q > 1$. 
	It follows from \eqref{e11} that
	\begin{equation}\label{e12}
		\|\na \ue\|_{L^q(Q_T)}^q \leq C(K) + \delta(K)\left(\|\ue\|_{L^r(Q_T)}^{r(p-q)/p} + \|\na \ue\|_{L^q(Q_T)}^{q/p}\|\ue\|_{L^r(Q_T)}^{r(p-1)/p} \right)
	\end{equation}
	with
	\begin{equation*}
		\delta(K) = 2\max\left\{C_0^{q/p}\Big(\sum_{n=K+1}^\infty \dfrac{1}{n^{r(p-q)/q}}\Big)^{q/p}; C_1^{q/p}\Big(\sum_{n=K+1}^\infty \dfrac{1}{n^{r(p-1)/q}}\Big)^{q/p} \right\}
	\end{equation*}
	with the property $\lim_{K\to\infty}\delta(K) = 0$ thanks to \eqref{cond_r}. From Young's inequality, and recalling that $q/p <q$, we have
	\begin{equation*}
		\|\na \ue\|_{L^q(Q_T)}^{q/p}\|\ue\|_{L^r(Q_T)}^{r(p-1)/p} \leq \frac 1p\|\na \ue\|_{L^q(Q_T)}^q + \frac{p-1}{p}\|\ue\|_{L^r(Q_T)}^{r}.
	\end{equation*}
	Therefore, \eqref{e12} implies
	\begin{equation}\label{e13}
		\begin{aligned}
		&\|\na \ue\|_{L^q(Q_T)}^q\\
		&\leq C(K) + \delta(K)\left[\|\ue\|_{L^r(Q_T)}^{r(p-q)/p} + \frac{p-1}{p}\|\ue\|_{L^r(Q_T)}^r + \frac 1p\|\na\ue\|_{L^q(Q_T)}^q\right]\\
		&\leq C(K) + \delta(K)\left[\frac qp + \frac{2p-q-1}{p}\|\ue\|_{L^r(Q_T)}^r + \frac 1p\|\na\ue\|_{L^q(Q_T)}^q\right]
		\end{aligned}
	\end{equation}
	where we used $\frac{r(p-q)}{p} = r - \frac{rq}{p} < r$ and the Young inequality $y^{r(p-q)/q} \leq \frac{p-q}{p}y^r + \frac qp$ at the last step. We will show now that by choosing a suitable $r$ (which satisfies \eqref{cond_r}) we can estimate
	\begin{equation*}\label{e13_1}
		\|\ue\|_{L^r(Q_T)}^r \leq C(T,\beta)\|\na\ue\|_{L^q(Q_T)}^q
	\end{equation*}
	with $\beta$ is in \eqref{cond1}. Indeed, by setting
	\begin{equation}\label{chose_r}
		r = \frac{q(d+1)}{d},
	\end{equation}
	we have
	\begin{equation*}
		\frac{r(p-q)}{q} = \frac{(d+1)(p-q)}{d} > 1 \quad \text{ since } \quad q < p - \frac{d}{d+1},
	\end{equation*}
	thus \eqref{cond_r} is satisfied. Note that from \eqref{chose_r} we also have $r < q^* = \frac{dq}{d-q}$. Therefore, we can use the interpolation inequality with $\frac{1}{r} = \frac{\eta}{1} + \frac{1-\eta}{q^*}$, and $\sup_{t\in(0,T)}\|\ue\|_{L^1(\Omega_t)}\leq \beta$ to estimate
	\begin{equation}\label{e16}
		\begin{aligned}
		\|\ue\|_{L^r(Q_T)}^r &= \int_0^T\|\ue\|_{L^r(\Omega_t)}^rdt \leq \int_0^T\|\ue\|_{L^1(\Omega_t)}^{r\eta}\|\ue\|_{L^{q^*}(\Omega_t)}^{r(1-\eta)}dt\\
		&\leq \beta^{r\eta}\int_0^T\|\ue\|_{L^{q^*}(\Omega_t)}^{r(1-\eta)}dt.
		\end{aligned}
	\end{equation}
	From \eqref{chose_r}, we can easily check that $r(1-\eta)=q$. Therefore, \eqref{e16} yields
	\begin{equation}\label{e16_1}
		\|\ue\|_{L^r(Q_T)}^r\leq \beta^{r\eta}\|\ue\|_{\LQ{q}{q^*}}^q.
	\end{equation}
	By using Lemma \ref{embedding},
	\begin{equation}\label{e16_2}
		\|\ue\|_{\LQ{q}{q^*}}^q  \leq C_{\Omega,T}^q\int_0^T\|\na\ue\|_{L^{q}(\Omega_t)}^qdt = C_{\Omega,T}^q\|\na\ue\|_{L^q(Q_T)}^q.
	\end{equation}
	Combining \eqref{e13}, \eqref{e16_1} and \eqref{e16_2} leads to
	\begin{equation}\label{e16_3}
		\|\na\ue\|_{L^q(Q_T)}^q \leq C(K)+\delta(K)\left[\frac qp + \left(\frac{2p-q-1}{p}\beta^{r\eta}C_{\Omega,T}^q + \frac 1p \right)\|\na\ue\|_{L^q(Q_T)}^q\right].
	\end{equation}
	Recalling that $\lim_{K\to\infty} \delta(K) = 0$. We choose $K$ large enough to have
	\begin{equation*}
		\delta(K)\left(\frac{2p-q-1}{p}\beta^{r\eta}C_{\Omega,T}^q + \frac 1p \right) \leq \frac 12,
	\end{equation*}
	which, in combination with \eqref{e16_3}, implies
	\begin{equation*}
		\|\na\ue\|_{L^q(Q_T)}^q \leq 2\left(C(K) + \delta(K)\frac qp\right),
	\end{equation*}
	which is the desired estimate \eqref{e2_1}.
\end{proof}
In order to prove Lemma \ref{UniformBounds}, thanks to Lemma \ref{lem:inter}, it is sufficient to prove \eqref{cond1} and \eqref{cond2} for solutions to the approximate problem \eqref{approx}. These will be shown in the next consecutive lemmas.
\begin{lemma}\label{prove_cond1}
	There exists a constant $\beta = \beta\left(T,\|u_0\|_{L^1(\Omega_0)}, \|f\|_{L^1(Q_T)}\right)$ {\normalfont independent of $\eps$} such that for any solution to \eqref{approx}, the following holds
	\begin{equation*}
		\|\ue\|_{\LQ{\infty}{1}} \leq \beta.
	\end{equation*}
\end{lemma}
\begin{proof}
	Let $k \in \mathbb{R}^+$. We define the truncated function
	\begin{equation*}\label{T_k}
	T_k(z) = \begin{cases}
	z, &\text{ if } |z| \leq k,\\
	k, &\text{ if } z > k,\\
	-k, &\text{ if } z < -k.
	\end{cases}
	\end{equation*}
	It is clear that $T_k$ is a Lipschitz function, and if $\ue  \in \LW{p}{1}{p}$ then $T_k(\ue ) \in \LW{p}{1}{p}$ with 
	\begin{equation}\label{na-T_k}
	\nabla T_k(\ue ) = \chi_{\{|\ue | \leq k\}}\nabla \ue ,
	\end{equation}
	where $\chi_{\{|\ue | \leq k\}}$ is the characteristic function of the set $\{|\ue (x, t)|\leq k\}$. Define $S_k(z) = \int_0^zT_k(\tau)d\tau$. We will show the following weak chain rule 
	\begin{equation}\label{t0}
		\int_{Q_t}\pa_s\ue T_k(\ue)dxds = \int_{\Omega_s}S_k(\ue)dx\bigg|_{s=0}^{s=t}.
	\end{equation}
	Indeed, choose a smooth sequence $\{v^m\}\subset C^1([0,T];C^1_c(\Omega_t))$ such that $v^m \xrightarrow{m\to\infty} \ue$ in $L^p(0,T;W^{1,p}_0(\Omega_t))\cap C([0,T];L^p(\Omega_t))$ and $\pa_t v^m \xrightarrow{m\to\infty} \pa_t \ue$ in $L^{p'}(0,T;W^{-1,p'}(\Omega_t))$. Since $S_k'(z) = T_k(z)$ and $v^m$ is smooth, it holds
	\begin{align*}
		\int_{Q_t}\pa_sv^m T_k(v^m)dxds = \int_{Q_t}\pa_s(S_k(v^m))dxds &= \int_0^t\left[\frac{d}{ds}\int_{\Omega_s}S_k(v^m)dx - \int_{\pa\Omega_s}S_k(v^m)(\vv\cdot \nu)dS\right]ds\\ &=\int_{\Omega_s}S_k(v^m)dx\bigg|_{s=0}^{s=t}
	\end{align*}
	where we used $v^m|_{\pa\Omega_s} = 0$ an $S_k(0) = 0$ at the last step. Let $m\to \infty$, thanks to $v^m \to \ue$ in $C([0,T];L^p(\Omega_t))$ and $|S_k(z)| \leq C(1+|z|)$, it follows that \begin{equation}\label{t1}\int_{\Omega_s}S_k(v^m)dx\bigg|_{s=0}^{s=t} \to \int_{\Omega_s}S_k(\ue)dx\bigg|_{s=0}^{s=t}.
	\end{equation}
	For the left hand side, we estimate
	\begin{equation}\label{t2}
	\begin{aligned}
		&\left|\int_{Q_t}(\pa_sv^mT_k(v^m) - \pa_s\ue T_k(\ue))dxds \right|\\
		&\leq \int_{Q_t}\left|\pa_sv^m - \pa_s\ue \right||T_k(v^m)|dxds+ \left|\int_{Q_t}(T_k(v^m)-T_k(\ue))\pa_s\ue dxds \right|\\
		&=: (I) + (II).
	\end{aligned}
	\end{equation}
	Since
	\begin{equation*}
		(I)\leq \|\pa_sv^m - \pa_s\ue\|_{L^{p'}(0,T;W^{-1,p'}(\Omega_t))}\|T_k(v^m)\|_{L^p(0,T;W^{1,p}_0(\Omega_t))}
	\end{equation*}
	and $\{T_k(v^m)\}_{m\geq 1}$ is bounded in $L^p(0,T;W^{1,p}_0(\Omega_t))$, 
	we have $\lim_{m\to\infty}(I) = 0$. For $(II)$ it follows from $v^m\to \ue$ a.e. in $Q_T$ and $T_k$ is continuous that $T_k(v^m)\to T_k(\ue)$ a.e. in $Q_T$. 
	By combining this with $\{T_k(v^m)\}$ is bounded in $L^p(Q_T)$, we obtain $T_k(v^m) \rightharpoonup T_k(\ue)$ weakly in $L^p(Q_T)$ (see e.g. \cite[Lemma 8.3]{robinson2001infinite}). Now since $T_k(\ue)\in L^p(0,T;W^{1,p}_0(\Omega_t))$, it yields $T_k(v^m) \rightharpoonup T_k(\ue)$ weakly in $L^p(0,T;W^{1,p}_0(\Omega_t)$ and therefore $\lim_{m\to\infty}(II) = 0$. By combining \eqref{t1} and \eqref{t2}, we obtain the desired relation \eqref{t0}.
	
	Choosing $\phi=T_k(\ue )$ as test function for \eqref{approx} and using \eqref{t0}, we get, for $0<t\leq T$,
	\begin{equation}\label{test-T_k}
		\begin{aligned}
			\int_{\Omega_s}S_k(\ue )dx\biggr|_{s=0}^{s=t} & + \int_{Q_t}a(x,s, \nabla \ue )\nabla T_k(\ue )dxds \\
			&+\int_{Q_t}\div(\ue\vv)T_k(\ue )dxds +\int_{Q_t} \Ge(\ue , \nabla \ue )T_k(\ue )dxds\\
			&= \int_{Q_t}\fe T_k(\ue )dxds.
		\end{aligned}
	\end{equation}
	Note that the boundary terms vanish due to the homogeneous Dirichlet boundary condition, which consequently implies that $S(\ue) = 0$ on the boundary.
	
	From \eqref{A3} and \eqref{na-T_k}, we have
	\begin{equation}\label{a}
	\begin{aligned}
	\int_{Q_t}a(x,s, \nabla \ue )\nabla T_k(\ue )dxds &= \int_{Q_t}\chi_{\{|\ue | \leq k\}}a(x,s, \nabla \ue ) \nabla \ue  dxds\\
	&\geq \alpha \int_{Q_t}\chi_{\{|\ue | \leq k\} } |\nabla \ue |^p dxds.
	\end{aligned}
	\end{equation}
	Applying intergration by parts for penultimate term on the left hand side of \eqref{test-T_k}, we obtain
	\begin{equation}\label{e17}
	\begin{aligned}
	\int_{Q_t}\div (\ue  \vv)T_k(\ue )dxds = - \int_{Q_t}\ue  \vv \nabla T_k(\ue )dxds= -\int_{Q_t}\chi_{\{|\ue | \leq k\}}\ue  \vv \nabla \ue  dxds.
	\end{aligned}
	\end{equation}
	Combining \eqref{a}-\eqref{e17} with \eqref{test-T_k}, we get
	\begin{equation}\label{e18}
	\begin{aligned}
	\int_{\Omega_s}S_k(\ue )dx\biggr|_{s=0}^{s=t} &+\alpha\int_{Q_t}\chi_{\{|\ue | \leq k\}}|\nabla \ue |^p dxds +\int_{Q_t}\Ge(\ue,\na\ue)T_k(\ue)dxds\\
	&\leq \int_{Q_t}|\fe T_k(\ue )|dxds + \int_{Q_t}\chi_{\{|\ue | \leq k\}}\ue  \vv \nabla \ue dxds.
	\end{aligned}
	\end{equation}
	Applying Young's inequality for last term in right-hand side above, we have
	\begin{equation}\label{e19}
	\begin{aligned}
	\int_{Q_t}\chi_{\{|\ue | \leq k\}} \ue  \vv \nabla \ue  dxds &\leq \|\vv\|_{\infty}\int_{Q_t}|\chi_{\{|\ue | \leq k\}} \ue||\na\ue|dxds\\ &\leq \frac{\alpha}{2}\int_{Q_t}\chi_{\{|\ue | \leq k\}}|\nabla \ue |^p dxds + C(\alpha,\|\vv\|_{\infty})\int_{Q_t}\chi_{\{|\ue | \leq k\}} |\ue |^{p'}dxds\\
	&\leq \frac{\alpha}{2} \int_{Q_t}\chi_{\{|\ue | \leq k\}}|\nabla \ue |^p dxdt +C(\alpha,\|
	\vv\|_{\infty})|k|^{p'}|Q_T|\\
	&=\frac{\alpha}{2} \int_{Q_t}\chi_{\{|\ue | \leq k\}}|\nabla \ue |^p dxdt +C(T, \alpha, \|\vv\|_{\infty},k),
	\end{aligned}
	\end{equation}
	where $\frac{1}{p} + \frac{1}{p'} = 1$. Since $|T_k(\ue)| \leq k$, 
	\begin{equation}\label{e20}
		\int_{Q_t}|\fe T_k(\ue)|dxdt \leq k\|\fe\|_{L^1(Q_T)}.
	\end{equation}
We remark that $u_\varepsilon T_k(\ue ) \geq 0$, combining with \eqref{G1}, we have $\Ge(\ue,\na\ue)T_k(\ue) \geq 0$. Therefore, inserting \eqref{e19} and \eqref{e20} into \eqref{e18} gives
	\begin{equation}\label{e23}
		\sup_{t\in(0,T)}\|S_k(\ue)(t)\|_{L^1(\Omega_t)} \leq \|S_k(u_{\eps,0})\|_{L^1(\Omega_0)} + k\|\fe\|_{L^1(Q_T)} + C(T,\alpha,\|\vv\|_\infty,k).
	\end{equation}	
	We set $k= 1$ in \eqref{e23}. Note that $0\leq S_1(z) \leq |z|$ and recall \eqref{increasing}, we get
	\begin{equation*}\label{e25}
	\begin{aligned}
	\sup_{t\in(0,T)}\int_{\Omega_t}S_1(\ue )(t)dx\leq \|u_{0}\|_{L^1(\Omega_0)} + k\|f\|_{L^1(Q_T)} + C(T,\alpha,\|\vv\|_\infty,1).
	\end{aligned}
	\end{equation*}
	Therefore, by using $\ue = S_1(\ue)$ for $|\ue| \geq 1$,
	\begin{equation*}\label{e26}
	\begin{aligned}
	\sup_{t\in(0,T)}\|\ue\|_{L^1(\Omega_t)} &=\sup_{t\in(0,T)}\int_{\{x\in \Omega_t: \;|\ue |\leq 1\}}|\ue |dx + \sup_{t\in(0,T)}\int_{\{x\in \Omega_t: \;|\ue |\geq 1\}}|\ue |dx\\
	&\leq \sup_{t\in(0,T)}|\Omega_t| + \sup_{t\in(0,T)}\int_{\Omega_t}|S_1(\ue )|dx\\
	&\leq \sup_{t\in(0,T)}|\Omega_t|+\|u_{0}\|_{L^1(\Omega_0)} + k\|f\|_{L^1(Q_T)} + TC(T,\alpha,\|\vv\|_\infty,1)\\
	&=:\beta.
	\end{aligned}
	\end{equation*}
	This completes the proof of Lemma \ref{prove_cond1}.
\end{proof}

\begin{lemma}\label{prove_cond2}
	There exist positive constants $C_0, C_1$ {\normalfont independent of $\eps$ and $n\in\mathbb N$} such that the following estimate holds
	\begin{equation*}
		\int_{B_n}|\na\ue|^pdxdt \leq C_0 + C_1\int_{E_n}|\na\ue|dxdt \quad \text{ for all } \quad \eps>0 \text{ and all } n\in \mathbb N,
	\end{equation*}
	where $\ue$ is a solution to \eqref{approx}. 
\end{lemma}
\begin{proof}
For $n\in \mathbb N$, we define the function $\phi_n: \mathbb R\to \mathbb R$ as
\begin{equation}\label{phi_n}
\phi_n(z) = \begin{cases}
1, &\text{ if } z> n+1,\\
z-n, &\text{ if } n \leq z \leq n+1,\\
0, &\text{ if } -n < z < n,\\
-z-n, &\text{ if } -n-1 \leq z \leq -n,\\
1, &\text{ if } z \leq -n-1,
\end{cases}
\end{equation}
and we set $\Psi_n(z) = \int_0^z\phi_n(\tau)d\tau$. We note that $\phi_n$ is a Lipschitz function, and therefore $\ue \in \LW{p}{1}{p}$ implies $\phi_n(\ue ) \in \LW{p}{1}{p}$ with
\begin{equation*}\label{nabla-phi}
\nabla \phi_n(\ue ) = \chi_{B_n}\nabla \ue ,
\end{equation*}
$\chi_{B_n}$ denoting the characteristic function of the set $B_n = \{(x,t) \in Q_T: n \leq |\ue (x,t)| \leq n+1\}$ defined in \eqref{def_BnEn}.
We now take $\phi_n(\ue ) \in \LW{p}{1}{p}$ as test function for \eqref{approx} to get, by using a weak chain rule similar to \eqref{t0},
\begin{equation}\label{test-phi_n_1}
\begin{aligned}
\int_{\Omega_T}&\Psi_n(\ue)(T)dx  +\int_{Q_T}a(x,t, \nabla \ue ) \nabla \phi_n(\ue )dxdt+\int_{Q_T}\div(\ue  \vv)\phi_n(\ue )dxdt\\
&+\int_{Q_T} \Ge(\ue , \nabla \ue )\phi_n(\ue )dxdt  \leq \int_{\Omega_0}\Psi_n(u_{0,\eps})dx +  \int_{Q_T}|\fe\phi_n(\ue )|dxdt.
\end{aligned}
\end{equation}
From \eqref{A3}, we obtain
\begin{equation}\label{e27}
\begin{aligned}
\int_{Q_T}a(x,t, \nabla \ue )\nabla \phi_n(\ue )dxdt &= \int_{Q_T}\chi_{B_n}a(x,t, \nabla \ue ) \nabla \ue  dxdt\\
&\geq \alpha \int_{Q_T}\chi_{B_n} |\nabla \ue |^p dxdt\\
&= \alpha\int_{B_n}|\na\ue|^pdxdt.
\end{aligned}
\end{equation}
The penultimate term on the left hand side of \eqref{test-phi_n_1} can be rewritten as
\begin{equation}\label{e28}
\begin{aligned}
\int_{Q_T}\div (\ue  \vv)\phi_n(\ue )dxdt =  \int_{Q_T}(\na\ue\cdot\vv + \ue\div\vv) \phi_n(\ue )dxdt.
\end{aligned}
\end{equation}
Combining \eqref{e27}-\eqref{e28} with \eqref{test-phi_n_1}, and the fact that $\Psi_n$ is nonnegative, we get
\begin{equation}\label{e29}
\begin{aligned}
&\alpha\int_{B_n}|\nabla \ue |^p dxdt + \int_{Q_T}\Ge(\ue,\na\ue)\phi_n(\ue)dxdt \\
&\leq \int_{\Omega_0}\Psi_n(u_{0,\eps})dx+ \int_{Q_T}|\fe||\phi_n(\ue )|dxdt\\
&\quad + \int_{Q_T}|\na\ue \cdot \vv||\phi_n(\ue)| dxdt + \int_{Q_T}|\div\vv||\ue||\phi_n(\ue)|dxdt\\
&\leq \int_{\Omega_0}\Psi_n(u_{0,\eps})dx + \|f\|_{L^1(Q_T)} + \|\vv\|_{\infty}\int_{Q_T}|\na\ue||\phi_n(\ue)|dxdt + \|\div\vv\|_{\infty}T\beta
\end{aligned}
\end{equation}
where we used $\sup_{t\in(0,T)}\|\ue\|_{L^1(\Omega_t)}\leq \beta$ at the last step.  Using $|\Psi_n(z)| \leq |z|$ and $\textrm{supp}(\phi_n) \subset (-\infty, -n] \cup [n,\infty)$, we can estimate
\begin{equation*}
	\left|\int_{\Omega_0}\Psi_n(u_{0,\eps})dx \right| \leq \|u_{0,\eps}\|_{L^1(\Omega_0)} \leq \|u_0\|_{L^1(\Omega_0)}
\end{equation*}
and
\begin{align*}
	\int_{Q_T}|\na\ue||\phi_n(\ue)|dxdt &\leq \int_{B_n}|\na\ue|dxdt + \int_{E_n}|\na\ue|dxdt\\
	&\leq \frac{\alpha}{2}\int_{B_n}|\na\ue|^pdxdt + C(\alpha)|Q_T| + \int_{E_n}|\na\ue|dxdt,
\end{align*}
where Young's inequality was applied at the last step. Inserting these estimates into \eqref{e29} yields 
\begin{equation}\label{est_1}
\begin{aligned}
	&\alpha \int_{B_n}|\na\ue|^pdxdt + 2\int_{Q_T}\Ge(\ue,\na\ue)\phi_n(\ue)dxdt\\
	&\leq C\left(\|u_0\|_{L^1(\Omega_0)}, \|f\|_{L^1(Q_T)}, \alpha, \beta, \vv, T\right) + 2\|\vv\|_{\infty}\int_{E_n}|\na\ue|dxdt,
\end{aligned}
\end{equation}
which implies the desired estimate and therefore completes the proof of Lemma \ref{prove_cond2}.
\end{proof}
\begin{proof}[Proof of Lemma \ref{UniformBounds}]
	The proof of Lemma \ref{UniformBounds} is an immediate consequence of Lemmas \ref{lem:inter}, \ref{prove_cond1} and \ref{prove_cond2}. 
\end{proof}
\subsection{Uniform bounds of the nonlinearity}
\begin{lemma}\label{bound_non}
	Let $\ue$ be a solution of \eqref{approx}. Then the following estimate holds
	\begin{equation}\label{bound_non_1}
		\|\Ge(\ue,\na\ue)\|_{L^1(Q_T)} \leq K
	\end{equation}
	where $K$ is {\normalfont independent of $\eps$}. 
\end{lemma}
\begin{proof}
	To prove \eqref{bound_non_1} we fix $n\in \mathbb N$ and write
	\begin{equation*}
		\|\Ge(\ue,\na\ue)\|_{L^1(Q_T)} \leq \int\limits_{\{|\ue| \leq n+1\}}|\Ge(\ue,\na\ue)|dxdt + \int\limits_{\{|\ue|\geq n+1\}}|\Ge(\ue,\na\ue)|dxdt =: G_1 + G_2.
	\end{equation*}
	Recall the function $\phi_n$ defined in \eqref{phi_n}, we have $\phi_n(z) = 1$ for $z\geq n+1$. Therefore, by using the fact that $\phi_n(\ue)\Ge(\ue,\na\ue) \geq 0$ thanks to \eqref{G1},
	\begin{align*}
		G_2 &= \int_{\{|\ue|\geq n+1\}}\phi_n(\ue)\Ge(\ue,\na\ue)dxdt\\
		&\leq \int_{Q_T}\phi_n(\ue)\Ge(\ue,\na\ue)dxdt\\
		&\leq C + \|\vv\|_{\infty}\int_{E_n}|\na\ue|dxdt \quad (\text{by using } (\ref{est_1}))\\
		&\leq C + \|\vv\|_{\infty}\|\na\ue\|_{L^1(Q_T)}\\
		&\leq C \quad (\text{by applying Lemma }\ref{UniformBounds} \text{ for } q= 1).
	\end{align*}
	Therefore, $G_2$ is bounded uniformly in $\eps$. We now estimate $G_1$ by using the assumption \eqref{G2}
	\begin{equation}\label{est_G1}
	\begin{aligned}
		G_1&\leq \int_{\{|\ue| \leq n+1\}}h(\ue)(\gamma(x,t) + |\na\ue|^\sigma)dxdt\\
		&\leq h(n+1)\|\gamma\|_{L^1(Q_T)} + h(n+1)\int_{\{|\ue|\leq n+1 \}}|\na\ue|^\sigma dxdt.
	\end{aligned}
	\end{equation}
	Now, recalling $B_j = \{(x,t): j\leq \ue(x,t) \leq j+1 \}$,
	\begin{align*}
		\int_{\{|\ue|\leq n+1\}}|\na\ue|^\sigma dxdt &= \sum_{j=0}^n\int_{B_j}|\na\ue|^\sigma dxdt\\
		&\leq \sum_{j=0}^n|B_j|^{\frac{p-\sigma}{p}}\left(\int_{B_j}|\na\ue|^pdxdt \right)^{\frac{\sigma}{p}}\\
		&\leq |Q_T|^{\frac{p-\sigma}{p}}\sum_{j=0}^n\left(C + \frac{2\|\vv\|_{\infty}}{\alpha}\int_{E_j}|\na\ue|dxdt \right)^{\frac{\sigma}{p}} \quad (\text{using }(\ref{est_1}))\\
		&\leq |Q_T|^{\frac{p-\sigma}{p}}(n+1)\left(C + \frac{2\|\vv\|_{\infty}}{\alpha}\|\na\ue\|_{L^1(Q_T)} \right)^{\frac{\sigma}{p}}\\
		&\leq C(n,T)
	\end{align*}
	where we used Lemma \ref{UniformBounds} with $q = 1$ at the last step. From this and \eqref{est_G1}, it follows that $G_1$ is bounded uniformly in $\eps>0$. Thus \eqref{bound_non_1} is proved.
\end{proof}

\section{Proof of Theorem \ref{thm:main}}\label{proof}
The uniform bounds in Section \ref{approximate} imply that there exists a subsequence of $\{\ue\}_{\eps>0}$ such that
\begin{equation*}
	\ue \rightharpoonup u \quad \text{ weakly in } \quad \LW{q}{1}{q} \quad \text{ for all } \quad 1<q<p - \frac{d}{d+1}.
\end{equation*}
This limit function $u$ is a candidate for a weak solution to \eqref{e1}, but the weak convergence is far from enough to show that it is the case. We need convergence in stronger topologies, especially to pass to the limit for the nonlinearities. We start with a pointwise and $L^1$-convergence.
\begin{lemma}\label{L1convergence}
	Let $\{\ue\}_{\eps>0}$ be solutions to \eqref{approx}. Then there exists a subsequence of $\{\ue\}_{\eps>0}$ (not relabeled) such that
	\begin{equation*}
		\ue \to u \quad \text{ strongly in } L^{s}(Q_T) \quad \text{ for all } \quad 1\leq s < p-\frac{d}{d+1}.
	\end{equation*}
\end{lemma}
To prove Lemma \ref{L1convergence}, we need an Aubin-Lions lemma for the case of moving domains. A similar lemma was recently shown in \cite{Mou16}, but it is not directly applicable to our case. Therefore, a new version is necessary.
\begin{lemma}[An Aubin-Lions lemma in moving domains]\label{AL-lemma}
	Let $1\leq q <+\infty$ and $\{u_n\}_{n\geq 1}$ be a sequence which is bounded in $L^q(0,T;W_0^{1,q}(\Omega_t))$. Moreover, for any smooth function $\psi\in \mathcal{D}(Q_T)$ it holds
	\begin{equation}\label{time_derivative}
	\left|\int_{Q_T} u_n\partial_t \psi dxdt \right| \leq C\sup_{t\in (0,T)}\|\psi\|_{H^m(\Omega_t)}
	\end{equation}
	for some $m\in \mathbb N$. Then $\{u_n\}_{n\geq 1}$ is precompact in $L^1(Q_T)$, and when $q > 1$ then $\{u_n\}_{n\geq 1}$ is precompact in $L^{s}(Q_T)$ for all $1\leq s<q$.
\end{lemma}
\begin{remark}
	In \cite{Mou16}, instead of \eqref{time_derivative}, the following stronger condition was imposed
	\begin{equation*}
		\left|\int_{Q_T} u_n  \partial_t \psi dxdt \right| \leq C\sum_{|\alpha|\leq m}\|\partial_{\alpha}\psi\|_{\LQ{2}{2}}.
	\end{equation*}
	In our case, due to the fact that the right hand side belongs only $L^1(Q_T)$, it seems that \eqref{time_derivative} is unavoidable.
\end{remark}
\begin{proof}[Proof of Lemma \ref{AL-lemma}]
	Though Lemma \ref{AL-lemma} is an improved version of that in \cite{Mou16}, its proof still follows closely from the ideas therein with some suitable changes. We therefore postpone it and provide the full technical proof in the Appendix \ref{appendix2}.
\end{proof}
%
We can now apply the Aubin-Lions lemma to prove Lemma \ref{L1convergence}.
	\begin{proof}[Proof of Lemma \ref{L1convergence}]
		Thanks to Lemma \ref{UniformBounds}, we only need to check the condition \eqref{time_derivative}. First, we choose $m\in \mathbb N$ such that $H^{m-1}(\Omega_t)\subset L^\infty(\Omega_t)$ for all $t\in [0,T]$. Moreover, using similar arguments to Lemma \ref{embedding} we deduce that there exists a constant $C = C(\vv,T)$ such that
		\begin{equation*}
			\|v\|_{L^r(\Omega_t)} \leq C\|v\|_{H^{k}(\Omega_t)} \quad \text{ for all } \quad t\in [0,T]
		\end{equation*}
		for $k\in \{m-1,m\}$ and for any $r\in[1,\infty]$. Now, we multiply the approximating problem \eqref{approx}
		by $\psi\in \mathcal D(Q_T)$ and then integrate on $Q_T$ to get
		\begin{equation}\label{est_2}
		\begin{aligned}
			\int_{Q_T} \ue \partial_t\psi dxdt &= -\int_{Q_T} a(x,t,\na\ue)\cdot \na\psi dxdt - \int_{Q_T}[\na\ue \cdot \vv + \ue\div \vv]\psi dxdt\\
			&- \int_{Q_T} \Ge(\ue,\na\ue)\psi dxdt + \int_{Q_T} f_\varepsilon \psi dxdt.
		\end{aligned}
		\end{equation}
		From the assumption \eqref{A2} of $a$ we have
		\begin{equation*}
			\begin{aligned}
			\left|\int_{Q_T} a(x,t,\na\ue)\cdot \na\psi dxdt \right| &\leq \int_{Q_T}|\varphi||\na\psi|dxdt + K\int_{Q_T}|\na\ue|^{p-1}|\na\psi|dxdt\\
			&\leq \|\varphi\|_{L^{p'}(Q_T)}\|\nabla \psi\|_{L^p(Q_T)} + K\|\na\ue\|_{L^q(Q_T)}^{p-1}\|\na \psi\|_{L^{\frac{q}{q-p+1}}(Q_T)}\\
			&\leq C\sup_{t\in (0,T)}\|\psi\|_{H^m(\Omega_t)}.
			\end{aligned}
		\end{equation*}
		Similarly, by using the bounds in Lemmas \ref{UniformBounds} and \ref{bound_non} and $\|f_\varepsilon\|_{L^1(Q_T)} \leq \|f\|_{L^1(Q_T)}$ we get
		\begin{align*}
			&\left| \int_{Q_T} [\na\ue\cdot\vv + \ue\div \vv]\psi dxdt\right|\\
			&\leq \|\vv\|_{\infty}\|\na\ue\|_{L^q(Q_T)}\|\psi\|_{L^{q'}(Q_T)} + \|\div\vv\|_{\infty}\|\ue\|_{L^q(Q_T)}\|\na\psi\|_{L^{q'}(Q_T)}\\
			&\leq C(\|\psi\|_{L^{q'}(Q_T)} + \|\na\psi\|_{L^{q'}(Q_T)})\\
			&\leq C\sup_{t\in(0,T)}\|\psi\|_{H^m(\Omega_t)},
		\end{align*}
		and
		\begin{equation*}
			\left|\int_{Q_T} \Ge(\ue,\na\ue)\psi dxdt \right| \leq \|\Ge(\ue,\na\ue)\|_{L^1(Q_T)}\|\psi\|_{L^\infty(Q_T)} \leq  C\sup_{t\in (0,T)}\|\psi\|_{H^m(\Omega_t)},
		\end{equation*}
		and
		\begin{equation*}
			\left|\int_{Q_T} f_\varepsilon\psi dxdt \right| \leq \|f\|_{L^1(Q_T)}\|\psi\|_{L^\infty(Q_T)} \leq  C\sup_{t\in (0,T)}\|\psi\|_{H^m(\Omega_t)}.
		\end{equation*}
		Putting all these into \eqref{est_2} we get \eqref{time_derivative}, and therefore Lemma \ref{AL-lemma} implies the desired result of Lemma \ref{L1convergence} since $q < p - \frac{d}{d+1}$ is arbitrary.
\end{proof}

Due to the nonlinearities in the gradient in $a(x,t,\na\ue)$ and in $\Ge(\ue,\na\ue)$, we need also stronger convergence of the gradient.
\begin{lemma}[Almost everywhere convergence of the gradient]\label{gradients}
	Let $\{\ue\}_{\eps>0}$ be solutions of the approximate problem \eqref{approx}. Then the sequence $\{\nabla \ue \}_{\eps>0}$ converges to $\nabla u$ almost everywhere as $\varepsilon$ goes to zero.
	\end{lemma}
\begin{proof}
We will show that $\{\nabla \ue \}_\varepsilon$ is a Cauchy sequence in measure, i.e. for all $\mu > 0$
\begin{equation}\label{gra1}
\calA:= \text{meas}\{(x,t)\in Q_T: |\nabla u_{\varepsilon'} - \nabla \ue | > \mu\} \to 0,
\end{equation} 
as $\varepsilon', \varepsilon \to 0$. From this, after extracting a subsequence, we have the convergence $\na\ue \to \na u$ almost everywhere.

To prove \eqref{gra1}, we let $k>0$ and $\delta>0$ be chosen later and observe that
\begin{equation*}
	\calA \subset \calA_1 \cup \calA_2 \cup \calA_3 \cup \calA_4
\end{equation*}
where
\begin{equation*}\label{split}
\begin{aligned}
	&\calA_1 = \{(x,t)\in Q_T:\; |\na\ue| \geq k \},\\
	&\calA_2 = \{(x,t)\in Q_T: \; |\na u_{\eps'}| \geq k \},\\
	&\calA_3 = \{(x,t)\in Q_T:\; |\ue - u_{\eps'}| \geq \delta \},\\
	&\calA_4 = \{(x,t)\in Q_T: \; |\na\ue - \na u_{\eps'}| \geq \mu, |\na\ue| \leq k, |\na u_{\eps'}| \leq k \text{ and } |\ue - u_{\eps'}| \leq \delta \}.
\end{aligned}
\end{equation*}
We will estimate $\calA_i$, $i=1,\ldots, 4$ separately. Firstly, for $\calA_1$, by applying Lemma \ref{UniformBounds} with $q = 1$, we have
\begin{equation}\label{est_A1}
	|\calA_1| = \int_{\calA_1}dxdt \leq \frac{1}{k}\int_{\calA_1}|\na\ue|dxdt \leq \frac{1}{k}\|\na\ue\|_{L^1(Q_T)} \leq \frac{C}{k}
\end{equation}
for $C$ independent of $\eps$. Similarly,
\begin{equation}\label{est_A2}
|\calA_2| \leq \frac{1}{k}\|\na u_{\eps'}\|_{L^1(Q_T)} \leq \frac{C}{k}.
\end{equation}
For $\calA_3$, 
\begin{equation}\label{est_A3}
	|\calA_3| = \int_{\calA_3}dxdt \leq \frac{1}{\delta}\|\ue - u_{\eps'}\|_{L^1(Q_T)}.
\end{equation}
It remains to estimate $\calA_4$. Firstly, by using $T_{\delta}(\ue - u_{\eps'}) = \ue - u_{\eps'}$ on the set $\{(x,t): \;|\ue - u_{\eps'}|\leq \delta\}$, we have
\begin{equation}\label{h1}
\begin{aligned}
	|\calA_4| \leq \frac{1}{\mu}\int_{\{|\ue - u_{\eps'}| \leq \delta\}}|\na(\ue - u_{\eps'})|dxdt = \frac{1}{\mu}\int_{Q_T}\chi_{\{|\ue - u_{\eps'}| \leq \delta\}}|\na (\ue - u_{\eps'})|dxdt.
\end{aligned}
\end{equation}
Subtracting the equation \eqref{approx} for $\eps$ and $\eps'$, then taking $\phi = T_\delta(\ue - u_{\eps'})$ as a test function, we get
\begin{equation}\label{h2}
\begin{aligned}
&\int_{\Omega_T}S_\delta(\ue(x,T) - u_{\eps'}(x,T))dx + \int_{Q_T}(a(x,t,\na\ue) - a(x,t,\na u_{\eps'}))\na T_{\delta}(\ue - u_{\eps'})dxdt\\
&= \int_{\Omega_0}S_\delta(u_{0,\eps} - u_{0,\eps'})dx + \int_{Q_T}(\fe - f_{\eps'})T_{\delta}(\ue - u_{\eps'})dxdt\\
&- \int_{Q_T}((\ue - u_{\eps'})\div\vv + \vv\cdot\na(\ue - u_{\eps'}))T_{\delta}(\ue - u_{\eps'})dxdt\\
&- \int_{Q_T}(\Ge(\ue,\na\ue) - g_{\eps'}(u_{\eps'},\na u_{\eps'}))T_{\delta}(\ue - u_{\eps'})dxdt.
\end{aligned}
\end{equation}
Since $S_\delta$ is nonnegative and thanks to the assumption \eqref{A1}, the left hand side of \eqref{h2} is bounded below by
\begin{equation}\label{h3}
\begin{aligned}
	&\int_{Q_T}(a(x,t,\na\ue)-a(x,t,\na u_{\eps'})(\na \ue - \na u_{\eps'})\chi_{\{|\ue - u_{\eps'}| \leq \delta\}}dxdt\\
	&\geq C\int_{Q_T}\chi_{\{|\ue - u_{\eps'}| \leq \delta\}}\frac{1}{\Theta(x,t,\na\ue,\na u_{\eps'})}|\na\ue - \na u_{\eps'}|^{\theta}dxdt.
\end{aligned}
\end{equation}
For the right hand side of \eqref{h2}, we use $|T_\delta(z)| \leq \delta$ and $S_\delta(z) \leq \delta|z|$ to estimate
\begin{equation}\label{h4}
\begin{aligned}
	\text{ Right hand side of } (\ref{h2}) &\leq \delta\|u_{0,\eps} - u_{0,\eps'}\|_{L^1(\Omega_0)} + \delta\|f_\eps - f_{\eps'}\|_{L^1(Q_T)}\\
	&\; + \delta\|\div \vv\|_{\infty}\|\ue - u_{\eps'}\|_{L^1(Q_T)} + \delta\|\vv\|_{\infty}\|\na(\ue - u_{\eps'})\|_{L^1(Q_T)}\\
	&\; + \delta(\|\Ge(\ue,\na\ue)\|_{L^1(Q_T)} + \|g_{\eps'}(u_{\eps'},\na u_{\eps'})\|_{L^1(Q_T)})\\
	&\leq C\delta
\end{aligned}
\end{equation}
with $C$ independent of $\eps, \eps'$, and where we used the fact that $\{u_{0,\eps}\}$ is bounded in $L^1(\Omega_0)$, and all $\{\fe\}, \{\ue\}, \{\na \ue\}, \{\Ge(\ue,\na\ue)\}$ are bounded in $L^1(Q_T)$. Inserting \eqref{h3} and \eqref{h4} into \eqref{h2} gives
\begin{equation*}
	\int_{Q_T}\chi_{\{|\ue - u_{\eps'}| \leq \delta \}}\frac{1}{\Theta(x,t,\na\ue,\na u_{\eps'})}|\na\ue - \na u_{\eps'}|^{\theta}dxdt \leq C\delta.
\end{equation*}
By using H\"older's inequality, we have
\begin{align*}
\int_{Q_T}\chi_{\{|\ue - u_{\eps'}| \leq \delta \}}|\na\ue - \na\uee|dxdt
&\leq \left(\int_{Q_T}\chi_{\{|\ue - u_{\eps'}| \leq \delta \}}\frac{1}{\Theta(x,t,\na\ue,\na\uee)}|\na\ue - \na\uee|^{\theta}dxdt\right)^{\frac 1\theta}\\
&\quad \times \left(\int_{Q_T}\chi_{\{|\ue - u_{\eps'}| \leq \delta \}}\Theta(x,t,\na\ue,\na\uee)^{\frac{1}{\theta-1}}dxdt \right)^{\frac{\theta-1}{\theta}}\\
&\leq C\delta^{\frac 1\theta}\left(\int_{Q_T}\left[1+|\na\ue|^{\frac{\varrho}{\theta-1}} + |\na\uee|^{\frac{\varrho}{\theta-1}}\right]dxdt\right)^{\frac{\theta-1}{\theta}}
\end{align*}
where we used \eqref{beta} at the last step. Thanks to \eqref{varrho}, $\dfrac{\varrho}{\theta-1} < p - \dfrac{d}{d+1}$. Therefore, Lemma \ref{UniformBounds} implies
\begin{equation*}
	\int_{Q_T}\left[|\na\ue|^{\frac{\varrho}{\theta-1}} + |\na\uee|^{\frac{\varrho}{\theta-1}} \right]dxdt \leq C
\end{equation*}
and thus
\begin{equation*}
	\int_{Q_T}\chi_{\{|\ue-\uee|\leq \delta\}}|\na\ue - \na\uee|dxdt \leq C\delta^{\frac 1\theta}.
\end{equation*}
Inserting this into \eqref{h1} leads to
\begin{equation}\label{est_A4}
	|\calA_4| \leq \frac{C\delta^{\frac 1\theta}}{\mu}
\end{equation}
for a constant $C$ independent of $\eps, \eps'$.

Now let $\kappa>0$ be arbitrary. We first choose $k$ to be large enough so that \eqref{est_A1} and \eqref{est_A2} give
\begin{equation*}
	|\calA_1| + |\calA_2| \leq \frac{\kappa}{2}.
\end{equation*}
We next choose $\delta$ to be small enough ($k$ is now fixed) so that \eqref{est_A4} implies
\begin{equation*}
|\calA_4| \leq \frac{\kappa}{4}.
\end{equation*}
With $k$ and $\delta$ are fixed, since $\{\ue\}_{\eps>0}$ is a Cauchy sequence in $L^1(Q_T)$, thanks to Lemma \ref{L1convergence}, there exists $\eps_0 > 0$ such that, for all $\eps, \eps' \leq \eps_0$, \eqref{est_A3} implies
\begin{equation*}
	|\calA_3| \leq \frac{1}{\delta}\|\ue - u_{\eps'}\|_{L^1(Q_T)} \leq \frac{\kappa}{4}.
\end{equation*}
Therefore,
\begin{equation*}
	|\mathcal A| \leq |\calA_1| + |\calA_2| + |\calA_3| + |\calA_4| \leq \kappa \quad \text{ for all } \quad \eps, \eps' \leq \eps_0.
\end{equation*}
Thus \eqref{gra1} is proved.
\end{proof}
We are now ready to get strong convergence for the nonlinear term $a(x,t,\na u)$.
\begin{lemma}\label{convergence_a}
Let $\{\ue\}_{\eps>0}$ be solutions to the equation \eqref{approx}. Then, up to a subsequence,
\begin{equation*}
	a(x,t,\na\ue) \to a(x,t,\na u) \quad \text{ strongly in} \quad L^s(Q_T) \quad \text{for all} \quad 1\leq s < 1+\frac{1}{(p-1)(d+1)}.
\end{equation*}
\end{lemma}
\begin{proof}
	From Lemma \ref{gradients} and the fact that $a$ is continuous with respect to the third variable, we have
	\begin{equation}\label{est_a1}
		a(x,t,\na \ue) \to a(x,t,\na u) \quad \text{ almost everywhere } \quad \text{ in } \quad Q_T.
	\end{equation}
	By using assumption \eqref{A2} and Lemma \ref{UniformBounds} we have for any $1\leq s < 1+\frac{1}{(p-1)(d+1)}$,
	\begin{align}\label{est_a2}
		\|a(x,t,\na\ue)\|_{L^s(Q_T)}^s \leq C\int_{Q_T}|\varphi|^sdxdt + CK^{r}\int_{Q_T}|\na\ue|^{s(p-1)}dxdt \leq C
	\end{align}
	thanks to $s(p-1) < p-\frac{d}{d+1}$ and $s < 1 + \frac{1}{(p-1)(d+1)} < \frac{p}{p-1} = p'$. From \eqref{est_a1} and \eqref{est_a2}, the Egorov theorem implies that $\{a(x,t,\na\ue)\}_{\eps>0}$ is precompact in $L^s(Q_T)$ for all $1\leq s < 1+\frac{1}{(p-1)(d+1)}$, which finishes the proof of Lemma \ref{convergence_a}.
\end{proof}
Due to the subcritial growth of the nonlinearity $g$ in \eqref{G2}, its convergence cannot be obtained in the same way as for $a$ in Lemma \ref{convergence_a}. A different approach should be used, for which we need the following lemma.
\begin{lemma}\label{lem:g}
	Let $\{\ue\}_{\eps>0}$ be solutions to \eqref{approx}. Then
	\begin{equation*}\label{bound_non_2}
	\lim_{k\to\infty}\sup_{\eps>0}\int_{\{|\ue|\geq k\}}|\Ge(\ue,\na\ue)|dxdt = 0.
	\end{equation*}
\end{lemma}
\begin{proof}
	Since $T_k(\ue) = k$ for $\ue \geq k$, we have
	\begin{equation}\label{l1}
	\int_{\{|\ue|\geq k\}}|\Ge(\ue,\na\ue)|dxdt \leq \frac{1}{k}\int_{Q_T}\Ge(\ue,\na\ue)T_k(\ue)dxdt.
	\end{equation}
	By integrating \eqref{e18} on $(0,T)$ and using \eqref{e17} we obtain, in particular,
	\begin{equation}\label{l2}
	\begin{aligned}
	&\int_{Q_T}\Ge(\ue,\na\ue)T_k(\ue)dxdt\\
	&\leq \int_{\Omega_0}S_k(u_{0,\eps})dx + \int_{Q_T}|\fe T_k(\ue)|dxdt \\
	&\quad + \|\vv\|_{\infty}\int_{Q_T}|\na \ue||T_k(\ue)|dxdt + \|\div\vv\|_{\infty}\int_{Q_T}|\ue||T_k(\ue)|dxdt.
	\end{aligned}
	\end{equation}
	Let $M>0$. We then have the following useful estimates
	\begin{equation*}
		0 \leq S_k(z) \leq M^2 + k|z|\chi_{\{|z|>M\}} \quad \text{ and } \quad |T_k(z)| \leq M + k\chi_{\{|z|>M\}}.
	\end{equation*}
	Therefore,
	\begin{equation*}
		\int_{\Omega_0}S_k(u_{0,\eps})dx \leq M^2|\Omega_0| + k\int_{\{|u_{0,\eps}|>M\}}|u_{0,\eps}|dx,
	\end{equation*}
	\begin{equation*}
		\int_{Q_T}|\fe T_k(\ue)|dxdt \leq M\|\fe\|_{L^1(Q_T)} + k\int_{\{|\ue|>M \}}|\fe|dxdt,
	\end{equation*}
	\begin{equation*}
		\int_{Q_T}|\na\ue||T_k(\ue)|dxdt \leq M\|\na\ue\|_{L^1(Q_T)} + k\int_{\{|\ue|>M\}}|\na\ue|dxdt,
	\end{equation*}
	and
	\begin{equation*}
		\int_{Q_T}|\ue||T_k(\ue)|dxdt \leq M\|\ue\|_{L^1(Q_T)} + k\int_{\{|\ue|>M\}}|\ue|dxdt.
	\end{equation*}
	Using these estimates in \eqref{l1} and \eqref{l2}, we get
	\begin{equation}\label{l3}
	\begin{aligned}
		\int_{\{|\ue|\geq k\}}|\Ge(\ue,\na\ue)|dxdt &\leq \frac{M^2}{k}|\Omega_0| + \frac{CM}{k}\left(\|\fe\|_{L^1(Q_T)} + \|\ue\|_{\LW{1}{1}{1}}\right)\\
		&\quad +\int_{\{|u_{0,\eps}|>M\}}|u_{0,\eps}|dx + \int_{\{|\ue|>M \}}|\fe|dxdt\\
		&\quad + \|\vv\|_\infty\int_{\{|\ue|>M\}}|\na\ue|dxdt +  \|\div \vv\|_\infty\int_{\{|\ue|>M\}}|\ue|dxdt.
	\end{aligned}
	\end{equation}
	Due to the uniform bound of $\{\|\ue\|_{L^1(Q_T)}\}_{\eps>0}$ we have
	\begin{equation*}
		\lim_{M\to\infty}\sup_{\eps>0}|\{(x,t)\in Q_T:\; \ue(x,t)>M \}| \leq \lim_{M\to\infty}\frac{1}{M}\sup_{\eps>0}\|\ue\|_{L^1(Q_T)} = 0.
	\end{equation*}
	Therefore, from the fact that, as $\eps\to 0$, $\|u_{0,\eps} - u_0\|_{L^1(\Omega_0)} \to 0$, $\|\fe-f\|_{L^1(Q_T)} \to 0$ (by the constructions of $u_{0,\eps}$ and $\fe$), and $\|\ue-u\|_{L^1(Q_T)} \to 0$ (due to Lemma \ref{L1convergence}), and $\|\na\ue - \na u\|_{L^1(Q_T)} \to 0$ (due the fact that $\na\ue \to \na u$ almost everywhere, and $\|\na\ue\|_{L^q(Q_T)}$ is bounded uniformly in $\eps$ for some $q>1$), we imply that the last four terms on the right hand side of \eqref{l3} become arbitrary small as $M$ tends to infinity.
	
	Let $\kappa>0$ be arbitrary. We first choose $M$ large enough such that the sum of the last four terms on the right hand side of \eqref{l3} is smaller than $\kappa/2$. Then using the boundedness of $\|\fe\|_{L^1(Q_T)}$ and $\|\ue\|_{\LW{1}{1}{1}}$, there exists $k_0$ large enough, which is independent of $\eps$, such that for all $k\geq k_0$,
	\begin{equation*}
		\frac{M^2}{k}|\Omega_0| + \frac{CM^2}{k}\left(\|\fe\|_{L^1(Q_T)} + \|\ue\|_{\LW{1}{1}{1}}\right) \leq \frac{\kappa}{2}.
	\end{equation*}
	Therefore,
	\begin{equation*}
		\sup_{\eps>0}\int_{\{|\ue|\geq k\}}|\Ge(\ue,\na\ue)|dxdt \leq \kappa \quad \text{ for all } \quad k\geq k_0,
	\end{equation*}
	which proves the claim \ref{bound_non_2}.
\end{proof}
\begin{lemma}[Strong convergence of the first order terms]\label{lemg_varepsilon}
	As $\eps \to 0$, there exists a subsequence of $\{g_\varepsilon(\ue , \nabla \ue )\}$ that converges to $g(u, \nabla u)$ almost everywhere in $Q_T$ and strongly in $L^1(Q_T)$.
	\end{lemma}
\begin{proof}
	From Lemmas \ref{L1convergence} and \ref{gradients}, and the fact that $g$ is continuous with respect to the third and fourth variables, we have
	\begin{equation*}
		\Ge(\ue,\na\ue) = \frac{g(x,t,\ue,\na\ue)}{1+\eps|g(x,t,\ue,\na\ue)|} \to g(x,t,u,\na u) \; \text{ almost everywhere in } \; Q_T.
	\end{equation*}
	To show that this convergence is in fact strong in $L^1(Q_T)$-topology, it's sufficient to show that the set $\{\Ge(\ue,\na\ue)\}_{\eps>0}$ is weakly compact in $L^1(Q_T)$, or equivalently to show that
	\begin{equation}\label{weak_compact_G}
		\lim_{A\in \text{meas}( Q_T), |A|\to 0}\sup_{\eps>0}\int_{A}|\Ge(\ue,\na\ue)|dxdt = 0
	\end{equation}
	where $A\in \text{meas}(Q_T)$ means that $A\subset Q_T$ is a measurable subset of $Q_T$. Indeed, we have for any $k\in \mathbb N$,
	\begin{equation}\label{split_g}
		\int_{A}|\Ge(\ue,\na\ue)|dxdt = \int_{A\cap \{|\ue|\leq k\}}|\Ge(\ue,\na\ue)|dxdt + \int_{A\cap \{|\ue|\geq k\}}|\Ge(\ue,\na\ue)|dxdt.
	\end{equation}
	For the second part, we have
	\begin{equation}\label{est_g0}
		\int_{A\cap \{|\ue|\geq k\}}|\Ge(\ue,\na\ue)|dxdt \leq \int_{\{|\ue|\geq k\}}|\Ge(\ue,\na\ue)|dxdt
	\end{equation}	
	in which the right-hand side tends to $0$, as $k\to\infty$, uniformly in $\eps$, thanks to Lemma \ref{lem:g}. It remains to estimate the first part in \eqref{split_g}. From the assumption \eqref{G2}, we have
	\begin{equation}\label{est_g1}
	\begin{aligned}
		&\int_{A\cap \{|\ue|\leq k\}}|\Ge(\ue,\na\ue)|dxdt\\
		&\leq \int_{A\cap \{|\ue|\leq k\}}|g(\ue,\na\ue)dxdt\\
		&\leq h(k)\int_{A\cap \{|\ue|\leq k\}}\left(|\gamma(x,t)| + |\na\ue|^{\sigma}\right)dxdt\\
		&\leq h(k)\int_{A}|\gamma(x,t)|dxdt + h(k)\left(\int_{A\cap \{|\ue|\leq k\}}|\na\ue|^pdxdt\right)^{\frac{\sigma}{p}}|A|^{\frac{p-\sigma}{p}}
	\end{aligned}
	\end{equation}
	where we used H\"older's inequality and the obvious estimate $|A\cap \{\ue \leq k\}| \leq |A|$ at the last step. By using H\"older's inequality again we find
	\begin{equation}\label{est_g2}
		h(k)\int_{A}|\gamma(x,t)|dxdt \leq h(k)\|\gamma\|_{L^{p'}(Q_T)}|A|^{\frac{p'-1}{p'}}
	\end{equation}
	where we recall that $p' = \frac{p}{p-1}$. From Lemmas \ref{prove_cond2} and \ref{UniformBounds} (with $q = 1$) we can estimate
	\begin{equation}\label{est_g3}
	\begin{aligned}
		\int_{A\cap \{|\ue|\leq k\}}|\na\ue|^pdxdt &\leq \sum_{j=0}^{k}\int_{B_j}|\na\ue|^pdxdt\\
		&\leq \sum_{j=0}^{k}\left(C_0 + C_1\int_{E_j}|\na\ue|dxdt\right)\\
		&\leq \sum_{j=0}^k\left(C_0 + C_1\|\na\ue\|_{L^1(Q_T)}\right)\\
		&\leq C(k+1).
	\end{aligned}
	\end{equation}
	Inserting \eqref{est_g2} and \eqref{est_g3} into \eqref{est_g1} gives us
	\begin{equation}\label{est_g4}
		\int_{A\cap \{|\ue|\leq k \} }|\Ge(\ue,\na\ue)|dxdt \leq Ch(k)|A|^{\frac{p'-1}{p'}} + C(k+1)h(k)|A|^{\frac{p-\sigma}{p}}.
	\end{equation}
	Using \eqref{est_g0} and \eqref{est_g4} yields the desired estimate \eqref{weak_compact_G} which finishes the proof of Lemma \ref{lemg_varepsilon}.
\end{proof}
The last lemma is about the continuity in time.
\begin{lemma}\label{lem:time_continuity}
	The sequence $\{\ue\}_{\eps>0}$ is a Cauchy sequence in $C([0,T];L^1(\Omega_t))$ as $\eps\to 0$, and therefore $u\in C([0,T];L^1(\Omega_t))$.
\end{lemma}
\begin{proof}
	Let $\eps, \eps'>0$, subtracting the equations for $\ue$ and $\uee$ and taking $T_1(\ue - \uee)$ as the test function, we have
	\begin{equation*}
	\begin{aligned}
		&\int_{\Omega_t}S_1(\ue - \uee)(t)dx + \int_0^t\int_{\Omega_s}(a(x,s,\na\ue) - a(x,s,\na\uee))(\na\ue - \na\uee)\chi_{\{|\ue - \uee|\leq 1\}}dxds\\
		&\leq \int_{\Omega_0}S_1(u_{0,\eps} - u_{0,\eps'})dx - \int_0^t\int_{\Omega_s}(\vv\cdot \na(\ue - \uee) + (\ue - \uee)\div\vv)T_1(\ue - \uee)dxds\\
		&-\int_0^t\int_{\Omega_s}(\Ge(\ue,\na\ue) - g_{\eps'}(\uee,\na\uee))T_1(\ue-\uee)dxds + \int_0^t\int_{\Omega_s}(\fe - f_{\eps'})T_1(\ue - \uee)dxds.
	\end{aligned}
	\end{equation*}
	Using the assumption \eqref{A1} and $|T_1(z)| \leq 1$ and $S_1(z) \leq |z|$, we obtain
	\begin{equation*}
	\begin{aligned}
		&\sup_{t\in(0,T)}\int_{\Omega_t}S_1(\ue - \uee)(t)dx\\ &\leq m_{\eps,\eps'}:= \|u_{0,\eps}-u_{0,\eps'}\|_{L^1(\Omega_0)} + \|\vv\|_{\infty}\|\na \ue - \na \uee\|_{L^1(Q_T)} + \|\div\vv\|_{\infty}\|\ue-\uee\|_{L^1(Q_T)}\\
		&+\|\Ge(\ue,\na\ue) - g_{\eps'}(\uee,\na\uee)\|_{L^1(Q_T)} + \|\fe-f_{\eps'}\|_{L^1(Q_T)}
	\end{aligned}
	\end{equation*}
	where clearly $\lim_{\eps,\eps'\to 0}m_{\eps,\eps'} = 0$. Now by using $|z|\chi_{\{|z|>1\}}/2 \leq S_1(z)\chi_{\{|z|>1\}}$ and $|z|^2\chi_{\{|z|\leq 1\}}/2 \leq S_1(z)\chi_{\{|z| \leq 1\}}$, we can estimate
	\begin{equation*}
	\begin{aligned}
		\|\ue(t) - \uee(t)\|_{L^1(\Omega_t)} &\leq \int_{\{|\ue(t) - \uee(t)| \leq 1\}}|\ue(t) - \uee(t)|dx + \int_{\{|\ue(t) - \uee(t)| > 1\}}|\ue(t) - \uee(t)|dx\\
		&\leq |\Omega_t|^{1/2}\left(\int_{\{|\ue(t) - \uee(t)| \leq 1\}}|\ue(t) - \uee(t)|^2dx\right)^{1/2}+ 2\int_{\Omega_t}S_1(\ue - \uee)(t)dx\\
		&\leq |\Omega_t|^{1/2}\left(2\int_{\Omega_t}S_1(\ue - \uee)(t)dx\right)^{1/2} + 2\int_{\Omega_t}S_1(\ue - \uee)(t)dx\\
		&\leq \sqrt 2 |\Omega_t|^{1/2}\sqrt{m_{\eps,\eps'}} + 2m_{\eps,\eps'}.
	\end{aligned}
	\end{equation*}
	Hence,
	\begin{equation*}
		\lim_{\eps,\eps'\to 0}\sup_{t\in(0,T)}\|\ue(t) - \uee(t)\|_{L^1(\Omega_t)} = 0.
	\end{equation*}
	Therefore, $\{\ue\}_{\eps>0}$ is a Cauchy sequence in $C([0,T];L^1(\Omega_t))$, and thus $u\in C([0,T];L^1(\Omega_t))$.
\end{proof}
We are now ready to prove the main theorem of this paper.
\begin{proof}[Proof of Theorem \ref{thm:main}]
	Let $\phi\in C([0,T];W^{1,\infty}_0(\Omega_t))\cap C^1((0,T);L^{\infty}(\Omega_t))$ be the test function to the approximate problem. We have
	\begin{align*}
		\int_{\Omega_T}\ue(T)\phi(T) dx - \int_{Q_T}\ue \partial_t \phi dxdt + \int_{Q_T}a(x,t,\na\ue)\cdot\na\phi dxdt\\
		- \int_{Q_T}\ue \vv\cdot \na\phi dxdt + \int_{Q_T}\Ge(\ue,\na\ue)\phi dxdt\\
		= \int_{\Omega_0}u_{0,\eps}\phi(0)dx + \int_{Q_T}\fe \phi dxdt.
	\end{align*}
	By applying Lemmas \ref{L1convergence}, \ref{convergence_a}, \ref{lemg_varepsilon}, and \ref{lem:time_continuity}, and using \eqref{u0f}, we can pass to the limit as $\eps\to 0$ in all the terms to obtain that
	\begin{align*}
		\int_{\Omega_T}u(T)\phi(T) dx - \int_{Q_T}u \partial_t \phi dxdt + \int_{Q_T}a(x,t,\na u)\cdot\na\phi dxdt\\
		- \int_{Q_T}u \vv\cdot \na\phi dxdt + \int_{Q_T}g(u,\na u)\phi dxdt\\
		= \int_{\Omega_0}u_{0}\phi(0)dx + \int_{Q_T}f \phi dxdt
	\end{align*}
	or in other words, $u$ is a weak solution to \eqref{e1} on $(0,T)$. The proof of Theorem \ref{thm:main} is complete.
\end{proof}
\appendix
\section{Existence of approximate solutions}\label{appendix1}
This section is devoted to a proof of the global existence of a weak solution to the approximate system \eqref{approx}. We follow the ideas in \cite{CNO17}. 

We divide the time interval $[0;T]$ into $N\in \mathbb{N}$ smaller intervals $(t_j,t_{j+1})$ for $j=0,\ldots, N-1$ and define $\Delta:= \max_j|t_j-t_{j+1}|$. The points $t_j$ are chosen so that
\begin{enumerate}
	\item $\cup_{j=0,N-1}\Omega_{t_j}\times[t_j,t_{j+1}) \subset \widehat{\O}$,
	\item $\Omega_{t_j}$ has smooth boundary for all $j\in \{0,\cdots,N-1\}$,
	\item \eqref{A0}--\eqref{A1} hold for a.e. $x\in\Omega_{t_j}$ and for all $\xi\in\mathbb R^d$,
	\item $t_j$ are Lebesgue points of the map $[0,T]\ni t\mapsto a(\cdot,t,\cdot)\in L^1(\widehat{\Omega} \times B(0,R))$ for any $R>0$, where $B(0,R)\subset \R^d$ is the open ball centered at zero with radius $R$,
	\item $\Delta \to 0$ as $N\to\infty$.
\end{enumerate}
\newcommand{\wh}{\widehat}
We define the extended function $\wh{f}_\eps:\wh{Q} \to \R$ as $\wh{f}_\eps(x,t) = \fe(x,t)$ if $(x,t)\in \wh{Q}$ and $\wh{f}_\eps(x,t) = 0$ otherwise. Let us denote by $I_j=[t_j,t_{j+1})$. For each $j\in \{0,\ldots, N-1\}$ we consider the following equation
\begin{equation}\label{timeapprox}
\begin{cases}
\pa_t w^{(j)} - \div(a(x,t_j,\na w^{(j)})) + \div(w^{(j)}\vv) +   \Ge(w^{(j)}, \na w^{(j)}) = \wh{f}_\eps, &x\in \Omega_{t_j}, \; t\in I_j,\\
w^{(j)}(x,t) = 0, &x\in \pa\Omega_{t_j}, t\in I_j,\\
w^{(j)}(x,t_j) =\begin{cases} \lim_{t\to t_j^{-}}w^{(j-1)}(\zeta_{t_j - t_{j-1}}^{-1}(x), t), &x\in \Omega_{t_j}\cap \Omega_{t_{j-1}},\\
0&x\in \O_{t_j}\setminus \O_{t_{j-1}}.\end{cases}
\end{cases}
\end{equation}
If $t_0=0$ then we let $w^{(0)}(x,0)=u_{0,\varepsilon}(x)$. Note that we have the semigroup property $\zeta_{t+s} = \zeta_t\circ \zeta_s$ and the domains $\Omega_{t_j} = \zeta_{t_j - t_{j-1}}^{-1}(\Omega_{t_{j-1}})$ for $j=\overline{0,N-1}$.

\medskip
For any fixed $j\in \{0,\cdots,N-1\}$,  by classical results, see e.g. \cite{Lio69}, one obtains the existence of a solution $w^{(j)}\in L^1(\O_{t_j}\times I_j)\cap L^p(I_j;W^{1,p}_0(\O_{t_j}))$ with $\pa_tw^{(j)}\in L^{p'}(I_j;W^{-1,p'}(\O_{t_j}))$ of \eqref{timeapprox}. Moreover, $w^{(j)}\in C(I_j; L^1(\Omega_{t_j}))$. Denote by
$$ \Omega^\Delta:=\{(x,t): x\in\O_{t_j}, t\in I_j, j=0,\cdots,N-1\}=\bigcup\limits_{j=0,\cdots,N-1}\O_{t_j}\times I_j. $$
From \cite[Lemma 3.4]{CNO17}, we know that as $\Delta\to 0$, $\Omega^\Delta$ converges to $Q_T$ in Hausdorff sense, and as a consequence $\chi_{\Omega^\Delta}$ converges strongly to $\chi_{Q_T}$ in $L^s(\wh{Q})$ for all $s<\infty$. We now glue the solutions $w^{(j)}(x,t)$ of \eqref{timeapprox} together and define the approximate solutions 
\begin{equation*}\label{appsol}
w^{\Delta}: \wh{Q}\to \mathbb R\quad \text{with} \quad
w^{\Delta}(x,t) = \sum_{j=0}^{N-1}\chi_{I_j}(t)\chi_{\Omega_{t_j}}(x)w^{(j)}(x,t)
\end{equation*}
for $(x,t)\in \wh{Q}$. The function $w^{(j)}(x,t)\chi_{\O_{t_j}}(x)$ in the formulae above is the function which coincides with $w^{(j)}(x,t)$ in $\O_{t_j}$ and is equal to zero outside $\O_{t_j}$, that means $w^{\Delta}\equiv 0$ in $\wh{Q}\backslash \Omega^{\Delta}$.

In the sequel, we prove some {\it a priori} estimates of $w^\Delta$ which are independent of $\Delta$, thus allowing us to pass to the limit $\Delta \to 0$. In conclusion we have $w^\Delta \to v$ where $v$ is a solution to \eqref{approx}.
We are ready to give a proof of Theorem \ref{thm:approximate}.
\begin{proof}[Proof of Theorem \ref{thm:approximate}]
	The proof follows the ideas of \cite{CNO17}, so we only sketch some main steps here. For simplicity we set
$$ G_\varepsilon(u,\nabla u)=\di(u\vv)+g_\varepsilon(u,\nabla u). $$
	{\bf Step 1: Establishing a priori estimates of $w^{\Delta}$.} 
	
	First, we will prove $w^{\Delta}\in L^\infty(\O^\Delta)$ for any $t>0$. It is enough to prove the estimate in $\O_0\times (0, t_1)$.
	
Let $k \geq 2$ be arbitrary. By choosing $|w^{\Delta}|^{k-2}w^{\Delta}$ as a test function of \eqref{approx}, we have
	\begin{equation}\label{S.1}
	\begin{aligned}
	\dfrac{d}{dt}\int_{\Omega_0}|w^{\Delta}|^kdx&+k(k-1)\int_{\Omega_0}a(x,0,\nabla w^{\Delta})\cdot\nabla w^{\Delta}|w^{\Delta}|^{k-2}dx\\
	&=-k\int_{\Omega_0}G_\varepsilon(w^\Delta,\nabla w^\Delta)|w^\Delta|^{k-2}w^\Delta dx+k\int_{\Omega_0}f_\varepsilon|w^\Delta|^{k-2}w^\Delta dx.
	\end{aligned}
	\end{equation}
	From \eqref{A3}, equation \eqref{S.1} becomes
	\begin{equation*}\label{S.2}
	\begin{aligned}
	\dfrac{d}{dt}\int_{\Omega_0}|w^\Delta|^kdx&+k(k-1)\alpha\left(\dfrac{p}{p+k-2}\right)^p\int_{\Omega_0}|\nabla (w^\Delta)^{\frac{k+p-2}{p}}|^pdx\\
	&\leq k\int_{\Omega_0}|G_\varepsilon(w^\Delta,\nabla w^\Delta)||w^\Delta|^{k-1}dx+k\int_{\Omega_0}|f_\varepsilon||w^\Delta|^{k-1}dx.
	\end{aligned}
	\end{equation*}
	By integrating the inequality above from 0 to $t_1$ we have 
	$$\int_{\Omega_0}|w^\Delta(t)|^{k}dx\leq \int_{\Omega_0}|u_{0,\varepsilon}|^{k}dx+k\int_0^{t_1}\int_{\O_0}(|G_\varepsilon|+|f_\varepsilon|)|w^\Delta|^{k-1}dxdt.$$
	Fix $\xi>t_1$. By using H\"{o}lder and Young inequalities, we have
	\begin{align*}
	(1-&t_1\xi^{-\frac{k}{k-1}})\sup\limits_{t\in(0,t_1)}\int_{\O_0}|w^\Delta(t)|^kdx+\xi^{-\frac{k}{k-1}}\int_0^{t_1}\int_{\O_0}|w^\Delta(t)|^kdxdt\\
	&\leq\sup\limits_{t\in(0,t_1)}\int_{\O_0}|w^\Delta(t)|^kdx\leq \int_{\Omega_0}|u_{0,\varepsilon}|^{k}dx+k\int_0^{t_1}\int_{\O_0}(|G_\varepsilon|+|f_\varepsilon|)|w^\Delta|^{k-1}dxdt.\\
	&\leq\int_{\O_0}|u_{0,\varepsilon}|^kdx+\xi^k\int_0^{t_1}\int_{\O_0}(|G_\varepsilon|+|f_\varepsilon|)^kdxdt+\xi^{-\frac{k}{k-1}}\int_0^{t_1}\int_{\O_0}|w^\Delta(t)|^kdxdt.
	\end{align*}
	Hence
	{\small$$ \left(1-t_1\xi^{-\frac{k}{k-1}}\right)^{1/k}\sup\limits_{t\in(0,t_1)}\left(\int\limits_{\O_0}|w^\Delta(t)|^kdx\right)^{1/k}\leq \left(\int\limits_{\O_0}|u_{0,\varepsilon}|^kdx+\xi^k\int\limits_0^{t_1}\int\limits_{\O_0}(|G_\varepsilon|+|f_\varepsilon|)^kdxdt\right)^{1/k}. $$}
	Letting $k\to\infty$, we obtain
	\begin{equation*}\label{S6}
	\|w^\Delta(t)\|_{L^\infty(\O_0)}\leq \|u_{0,\varepsilon}\|_{L^\infty(\O_0)}+\xi(\|G_\varepsilon\|_{L^\infty(0,t_1;\O_0)}+\|f_\varepsilon\|_{L^\infty(0,t_1;\O_0)}),\,\forall \xi>t_1.
	\end{equation*}
	
	Second, by using the same arguments in \cite[Lemma 3.6 and 3.9]{CNO17}, we obtain two results respectively, for precisely there is some constant $C>0$ depending only on $Q_T$ such that
	\begin{equation*}\label{S7}
	\sum\limits_{j=0}^{N-1}\int_{t_j}^{t_{j+1}}\int_{\Omega_{t_j}}|\nabla w^\Delta|^p dx\leq C,
	\end{equation*}
	and let $0<j\leq N$ be fixed, then 
	\begin{equation}\label{S10}
	w^{(j)}_t\chi_{\O_{t_j}}\in L^{p'}(I_j; W^{-1,p'}(\O_{t_j})).
	\end{equation} 
	
	{\bf Step 2: Passing to the limits.} From the above estimates, and recalling that $w^{\Delta} \equiv 0$ in $\wh{Q}\backslash \Omega^{\Delta}$, we will show that there exists a subsequence of $\{w^\Delta\}_N$, also denoted by $\{w^\Delta\}_N$, such that
	\begin{itemize}
		\item[(i)] $w^\Delta\rightharpoonup v$ in $L^p(0,T;W^{1,p}(\wh{\O}))\cap L^\infty(\wh{Q})$,
		\item[(ii)] $a(x,t,\nabla w^{\Delta})\rightharpoonup a(x,t,\na v)$ in $L^{p'}(\wh{Q})$,
		\item[(iii)] $G_\varepsilon(w^\Delta,\nabla w^\Delta)\rightharpoonup G_\eps(v,\na v)$ in $L^{p'}(\wh{Q})$,
	\end{itemize}
	where $ a(x,t,\nabla w^{\Delta}):=\sum\limits_{j=0}^{N-1}\chi_{[t_j,t_{j+1})}(t)a(x,t_j,\nabla w^{(j)})\chi_{\O_{t_j}}(x)$. The limit (i) is straightforward. The limits (ii) and (iii) are proved in the following.
	\begin{itemize}
		\item \underline{Proof of the limit (ii)}. We only give the main ideas while refer the reader to \cite[Lemma 3.10 and 4.7]{CNO17} for more details.

		At first, we show that $\{w^{\Delta}|_{C}\}_{\Delta>0}$ is precompact in $L^1(C)$ where $C = (s_1,s_2)\times K$ is an open cylinder contained in $Q_T$ with $0\leq s_1<s_2 \leq T$ such that $\overline{C}\subset Q_T$. Since $\Omega^\Delta \to Q_T$, we can choose $N$ large enough such that $\overline{C}\subset \Omega^\Delta$. Moreover, if $(s_1,s_2)\cap I_j\ne \emptyset$ then $\overline{K}\subset\Omega_{t_j}$. Therefore,
		\begin{equation*}
			w^{\Delta}|_C(x,t) = \sum_{j=0}^{N-1}(\chi_{I_j}(t)\chi_{\Omega_{t_j}}(x)w^{(j)}(x,t))|_{C}.
		\end{equation*}
		It follows that $\{w^{\Delta}|_C\}$ is bounded in $L^p(s_1,s_2;W^{1,p}(K))$. We will now show
		\begin{equation}\label{limit}
		 \int_{s_1}^{s_2-h}\|w^\Delta(t+h)-w^\Delta(t)\|_{W^{-1,p'}(K)}dt\to 0,\quad\text{as }h\to 0^+ 	
		\end{equation}
		uniformly in $N$. For $z\in W^{1,p}_0(K)$ we can use zero extension to have $z\in W^{1,p}_0(\Omega_{t_j})$ and $\|z\|_{W^{1,p}_0(K)} = \|z\|_{W^{1,p}_0(\Omega_{t_j})}$ thanks to $\overline{K}\subset \Omega_{t_j}$. Thus we can estimate for $t\in I_j^h\cap (s_1,s_2-h)$ with $I_j^h:= [t_j,t_{j+1}-h)$
		\begin{equation}\label{t3}
		\begin{aligned}
			&\int_{I_j^h\cap (s_1,s_2-h)}\left|\int_{K}(w^{\Delta}(t+h) - w^{\Delta}(t))\cdot zdx\right|dt\\
			 &\leq\int_{I_j^h\cap(s_1,s_2-h)}\int_{t}^{t+h}\int_{K}|w^{(j)}_t\cdot z| dxdsdt\\
			&\leq h^{1/p}|I_j|\|w_t^{(j)}\|_{L^{p'}(I_j;W^{-1,p'}(\Omega_{t_j}))}\|z\|_{W^{1,p}_0(K)}.
		\end{aligned}
		\end{equation}
		For $t\in (I_j\backslash I_j^h)\cap (s_1,s_2-h)$,
		\begin{equation}\label{t4}
		\begin{aligned}
		\qquad&\int_{(I_j\backslash I_j^h)\cap (s_1,s_2-h)}\left|\int_{K}(w^{\Delta}(t+h) - w^{\Delta}(t))\cdot zdx\right|dt\\
		&\leq\int\limits_{(I_j\backslash I_j^h)\cap(s_1,s_2-h)}\left(\|w^{(j+1)}(t+h)\|_{W^{-1,p'}(\Omega_{t_{j+1}})} + \|w^{(j)}(t)\|_{W^{-1,p'}(\Omega_{t_j})} \right)\|z\|_{W^{1,p}_0(K)}dt\\
		&\leq h\left(\|w^{(j+1)}\|_{C(I_j;W^{-1,p'}(\Omega_{t_{j+1}}))}+\|w^{(j)}\|_{C(I_j;W^{-1,p'}(\Omega_{t_j}))}\right)\|z\|_{W^{1,p}_0(K)}.
		\end{aligned}	
		\end{equation}
		From \eqref{t3} and \eqref{t4} we can write
		\begin{gather*}
			\int_{s_1}^{s_2-h}\|w^{\Delta}(t+h)-w^{\Delta}(t)\|_{W^{-1,p'}(K)}dt\\ = \sum_{j=0}^{N-1}\int\limits_{I_j\cap(s_1,s_2-h)}\|w^{\Delta}(t+h)-w^{\Delta}(t)\|_{W^{-1,p'}(K)}dt
		\end{gather*}	
		to obtain the desired limit \eqref{limit}.
By using \cite[Theorem 5]{Si86}, $\{w^\Delta|_{C}\}_{\Delta >0}$ is relatively
compact in $L^1(C)$. For any compact subset of  $Q_T$, it can be covered by a finite number of open cylinders, then by applying a diagonal procedure, we deduce that the sequence $\{w^\Delta\}_{\Delta>0}$ is relatively compact in $L^1_{\text{loc}}(Q_T)$. Together with the uniform bound of $w^\Delta$ in $L^\infty(\O^\Delta)$ and $w^{\Delta}\equiv 0$ in $\wh{Q}\backslash \Omega^{\Delta}$, we obtain the following lemma.
	\begin{lemma}\label{strgcon}
		There exists a subsequence of $\{w^\Delta\}_{\Delta>0}$ which converges strongly in $L^1(\wh{Q})$.
	\end{lemma}
It's clear that $a(x,t, \na w^{\Delta}) \rightharpoonup \overline{a}$ in $L^{p'}(\wh{Q})$ for some $\overline{a}$. Moreover, we have the following result.
	\begin{lemma}
		Let $\phi$ be smooth and such that $\text{supp}\phi\subset \O^\Delta\cap ([0,T]\times \mathbb R^d)$. Then 
		\begin{equation*}\label{S0}
		\underset{N\to\infty}{\lim\sup}\int_0^T\int_{\O^\Delta_t} a(x,t,\nabla w^\Delta)\cdot \nabla w^\Delta\phi dxdt\leq \int_0^T\int_{\O_t}\overline{a}\cdot\nabla w\phi dxdt
		\end{equation*}
		where $\Omega_t^{\Delta} = \Omega_{t_j}$ if $t\in [t_j,t_{j+1})$, $j=0,\ldots, N-1$.
	\end{lemma}
	
	Then, we can now use the same arguments as in \cite[Lemma 4.8]{CNO17} to obtain $$\overline{a}(x,t,\nabla v)=a(x,t,\nabla v) \text{ a.e. in } Q_T,$$
	hence (ii).
	
	\medskip
	\item \underline{Proof of limit (iii)}.	From the boundedness of $G_\eps$, we have
	\begin{equation*}
		G_\eps(w^\Delta, \na w^{\Delta}) \rightharpoonup \overline{g} \quad \text{in}\quad L^{p'}(\wh{Q}).
	\end{equation*}
	It remains to prove $\overline{g}=G_\varepsilon(v,\nabla v)$ a.e. in $Q_T$. Since $G_\varepsilon$ is a continuous function with respect to $w$ and $\nabla w$, by classical results (see e.g. \cite{Lio69}), the sequence $G_\varepsilon(w^\Delta,\nabla w^\Delta)\rightharpoonup G_\varepsilon(v,\nabla v)$ in $L^1(\wh{Q})$ if we show that the sequence $\{\nabla w^\Delta\}$ converges to $\nabla v$ a.e. as $\Delta\to 0$. This property is obtained as we show that $\{\nabla w^\Delta\}$ is a Cauchy sequence in measure, see \cite{RE65}, i.e. for all $\mu>0$
	\begin{equation}\label{S19}
	\text{meas}\{(x,t)\in \wh{Q}:|\nabla w^\Delta-\nabla w^{\Delta'}|\geq\mu\}\to 0,\quad \text{as }\Delta, \Delta'\to 0.
	\end{equation}
	Let us denote by $\mathcal A$ the subset of $\wh{Q}$ involved in \eqref{S19}. Let $k>0$ and $\eta>0$, we have
	\begin{equation*}\label{S20}
	\mathcal A\subset \mathcal A_1\cup \mathcal A_2\cup \mathcal A_3\cup\mathcal A_4,
	\end{equation*}
	where
	\begin{align*}
	\mathcal A_1&=\{|\nabla w^\Delta|\geq k\},\\
	\mathcal A_2&=\{|\nabla w^{\Delta'}|\geq k\},\\
	\mathcal A_3&=\{|w^\Delta-w^{\Delta'}|\geq \eta\},\\
	\mathcal A_4&=\{|\nabla w^\Delta-\nabla w^{\Delta'}|\geq \mu, |\nabla w^\Delta|\leq k,|\nabla w^{\Delta'}|\leq k,|w^\Delta-w^{\Delta'}|\leq \eta\}.
	\end{align*}
	By repeating the arguments in Lemma \ref{gradients} we have \eqref{S19}.
	\end{itemize}
	{\bf Step 3: Recovery initial conditions.}
	
	We refer the reader to \cite[Proposition 4.9]{CNO17} for a proof that $v$ in {\bf Step 2} is a weak solution of problem \eqref{approx} and furthermore, $v(t)\to u_{0,\varepsilon}$ a.e. as $t\to 0$.
\end{proof}
\section{An Aubin-Lions lemma in moving domains}\label{appendix2}
This appendix provides a proof of the Aubin-Lions lemma in Lemma \ref{AL-lemma}. We follow the ideas from \cite{Mou16}. For any $\delta>0$, we write $\Omega_t^\delta = \{x\in \Omega_t: d(x,\partial\Omega_t)> \delta\}$ and 
\[
Q_T^\delta = \bigcup\limits_{t\in (0,T)}\Omega_t^\delta \times \{t\}.
\]
Let $\varphi: \mathbb R^d \to \mathbb R$ be a $C^\infty_c$ function such that
\begin{itemize}
	\item $\varphi$ is radially symmetric;
	\item $\mathrm{supp}(\varphi)\subset B(0,1)$;
	\item $\int_{\mathbb R^d}\varphi(x)dx = 1$.
\end{itemize}
\newcommand{\ve}{\varphi^\varepsilon}
We define the scaled mollifier as $\ve(x) = \varepsilon^{-d}\varphi(x/\varepsilon)$ and for any distribution $g\in \mathcal D'(Q_T)$ we have the convolution
\begin{equation*}
(g(\cdot,t)*\ve)(x) = \int_{\Omega_t}g(t,x-y)\ve(y)dy = \int_{\R^d}g(x-y,t)\ve(y)dy
\end{equation*}
defined on $\Omega_t^\varepsilon$, and consequently on $\Omega_t^\delta$ for all $\delta < \varepsilon$ by trivial extension.
\begin{lemma}\label{lemma_lim}
	Let $\delta > 0$. If $\{\na u_n\}$ is bounded in $L^p(Q_T)$ for some $p\geq 1$, then 
	\begin{equation}\label{lim}
	\lim_{\varepsilon\to 0}\sup_{n\geq 1}\|u_n*\ve - u_n\|_{L^p(Q_T^\delta)} = 0.
	\end{equation}
\end{lemma}
\begin{proof}
	By definition and the fact that $\int_{\R^d}\ve(y)dy = 1$ we have 
	\begin{equation*}
	\begin{aligned}
	|(u_n*\ve)(x,t) - u_n(x,t)| &\leq \int_{\R^d} |u_n(x-y,t)-u_n(x,t)||\ve(y)|dy\\
	&\leq \int_{\R^d}|y||\nabla u_n(s(x-y) + (1-s)x,t)||\ve(y)|dy\\
	&= \int_{|y| \leq \varepsilon}|y||\nabla u_n(s(x-y) + (1-s)x,t)||\ve(y)|dy.
	\end{aligned}
	\end{equation*}
	Integrating the above inequality over $Q_T^\delta$ and using the fact that $\{\nabla u_n\}$ is bounded in $L^p(Q_T)$, we get 
	\begin{equation*}
	\sup_{n\geq 1}\|u_n *\ve - u_n\|_{L^p(Q_T^\delta)} \leq C\int_{Q_T}\int_{|y|\leq \varepsilon}|y||\ve(y)|dydxdt,
	\end{equation*}
	and consequently \eqref{lim} as $\varepsilon\to 0$.
\end{proof}
\begin{lemma}[A local compactness lemma]\label{local}
	Assume all the conditions in Lemma \ref{AL-lemma} are fulfilled. Then there exists $\delta_0>0$ small enough such that for any $\delta\leq \delta_0$, $\{u_n\}$ is precompact in $L^s(Q_T^\delta)$ for all $1\leq s < p$.
\end{lemma}
\begin{proof}
	We first prove that for any fixed $\varepsilon<\delta_0$, the sequence $\{u_n*\ve \}_{n}$ is precompact in $L^1(Q_T^\delta)$. Indeed, using the condition \eqref{time_derivative}, and the fact that $\ve$ is radially symmetric we have
	\begin{equation*}
	\begin{aligned}
	\left|\int_{Q_T} \psi\partial_t(u_n*\ve)dxdt \right| & = \left|\int_{Q_T} (\psi *\ve)\partial_t u_n dxdt \right|\\
	&\leq C\sup_{t\in (0,T)}\|\psi*\ve\|_{H^m(\Omega_t)}\\
	&\leq C_\varepsilon\sup_{t\in(0,T)}\|\psi\|_{L^\infty(\Omega_t)}\\
	&\leq C_\varepsilon\|\psi\|_{L^\infty(Q_T)}.
	\end{aligned}
	\end{equation*}
	
	By duality, we get that $\{\partial_t (u_n*\ve) \}_{n}$ is bounded in $L^1(Q_T^\delta)$. From the assumption of $u_n$, we obtain that $\{u_n*\ve\}_{n}$ and $\{\nabla(u_n*\ve) \}_n$ are bounded in $L^1(Q_T^\delta)$. Therefore we have $\{u_n*\ve \}_n$ is bounded in $W^{1,1}(Q_T^\delta)$, and thus, by the compact embedding $W^{1,1}(Q_T^\delta) \hookrightarrow L^1(Q_T^\delta)$ we get that $\{u_n*\ve\}_n$ is a Cauchy sequence in $L^1(Q_T^\delta)$.
	
	By applying estimate \eqref{lim} in Lemma \ref{lemma_lim} and by writing 
	\begin{equation*}
	\begin{aligned}
	&\|u_n - u_m\|_{L^1(Q_T^\delta)}
	\\ &\leq \|u_n - u_n*\ve\|_{L^1(Q_T^\delta)} + \|u_n*\ve - u_m*\ve\|_{L^1(Q_T^\delta)} + \|u_m*\ve - u_m\|_{L^1(Q_T^\delta)}
	\end{aligned}
	\end{equation*}
	we obtain that $\{u_n\}_n$ is precompact in $L^1(Q_T^\delta)$. Using the boundedness of $\{u_n\}_n$ in $L^p(Q_T^\delta)$ and interpolation we obtain the precompactness of $\{u_n\}_n$ in $L^s(Q_T^\delta)$ for all $1\leq s < p$.
\end{proof}

We will also use the following result from \cite{Mou16}.
\begin{lemma}\cite[Proposition 8]{Mou16}\label{Mou}
	If $\{u_n\}_n$ and $\{\nabla u_n\}_n$ are bounded in $L^p(Q_T)$ and $\{u_n\}_n$ is precompact in $L^p(Q_T^\delta)$ for all $\delta < \delta_0$, then $\{u_n\}_n$ is precompact in $L^p(Q_T)$.
\end{lemma}

We have all the ingredients to prove Lemma \ref{AL-lemma}.
\begin{proof}[Proof of Lemma \ref{AL-lemma}]
	The proof follows directly from Lemmas \ref{local} and \ref{Mou} above.
\end{proof}

\medskip
\par{\bf Acknowledgements:} 
We would like to thank the referees for their helpful and constructive comments, which help to improve the presentation of this paper.
 
This research is funded by Thuyloi University Foundation for Science and Technology under grant number TLU.STF.19-04.\\
The third author is supported by the International Research Training Group IGDK 1754 "Optimization and Numerical Analysis for Partial Differential Equations with Nonsmooth Structures", funded by the German Research Council (DFG)  project number 188264188/GRK1754 and the Austrian Science Fund (FWF) under
grant number W 1244-N18. This work is partially supported by NAWI Graz.

\end{document}